\documentclass[11pt,a4paper]{amsart}
\usepackage{graphicx,amscd,amssymb}

\theoremstyle{plain}
\numberwithin{theorem}{section}
\newtheorem{main}{Main~Theorem}{}
\newtheorem{lemma}{Lemma}\numberwithin{lemma}{section}
\newtheorem{proposition}{Proposition}\numberwithin{proposition}{section}
\numberwithin{corollary}{section}
\theoremstyle{definition}
\newtheorem{definition}{Definition}\numberwithin{definition}{section}
\theoremstyle{remark} \theoremstyle{exam} \theoremstyle{ob}
\newtheorem{remark}{Remark}\numberwithin{remark}{section}
\numberwithin{notation}{section}
\newtheorem{example}{Example}\numberwithin{example}{section}
\numberwithin{ob}{section}
\numberwithin{nt}{section}
\numberwithin{equation}{section}

\begin{document}
\title[Spherical   Averages on  the  Nilpotent Groups]
    {Spherical  Averages  
   on  nilpotent Groups}
    \author{ Joonil Kim }

    \address{Department of Mathematics\\
            Yonsei University  \\
            Seoul 121, Korea}
   \email{jikim7030@yonsei.ac.kr}



\begin{abstract}
In this paper, we study the spherical maximal operator $\mathcal{M}_{S^{d-1}(A)}$, which is associated with the sphere $S^{d-1}$ embedded in the horizontal plane of some nilpotent group $\mathbb{G}_A$. Here, $\mathbb{G}_A$ is the  group identified with $\mathbb{R}^{d+1}$, and its group law is determined by the $d\times d$ real matrix $A$.
Specifically,  when  $A$ is a skew-symmetric matrix $J$, which is the case of the Heisenberg group,  it was  known that the spherical maximal operator 
$\mathcal{M}_{S^{d-1}(J)}$ is bounded on $L^p(\mathbb{G}_J)$ for $p$ in the range $(\frac{d}{d-1},\infty]$.
In this paper, we extend the results from the skew-symmetric matrix $J$ to a more general class of invertible matrices $A$. We aim to classify the range of $p$ for which the operator $\mathcal{M}_{S^{d-1}(A)}$ is bounded on $L^p(\mathbb{G}_A)$ based on the eigenvalues of $A$.
\end{abstract}
    \maketitle

\setcounter{tocdepth}{1}

\footnote{This research was supported by Basic Science Research Program through the National Research Foundation of Korea(NRF) funded by the Ministry of Education(2019R1H1A2039703)}

\tableofcontents

\section{Introduction} 

Let $M_d(\mathbb{R})$ denote the set of all $d\times d$ matrices with real entries, and let $G_d(\mathbb{R})$ and $SO_d(\mathbb{R})$ represent the subsets of invertible and orthogonal matrices in $M_d(\mathbb{R})$, respectively. Given a matrix $A \in M_d(\mathbb{R})$ and a measure $d\sigma$ along the unit sphere $S^{d-1} \subset \mathbb{R}^d$, we consider the maximal average $\mathcal{M}_{S^{d-1}(A)} f(x,x_{d+1})$ of a function $f \in L_{loc}(\mathbb{R}^{d+1})$ given by $\sup_{t>0}\mathcal{A}_{S^{d-1}(A)}f(x,x_{d+1},t)$ where
\begin{align}
\mathcal{A}_{S^{d-1}(A)}f(x,x_{d+1},t)&:=\int_{y\in S^{d-1}} f(x-ty, x_{d+1}-\langle Ax,ty\rangle) d\sigma(y). \label{11}
\end{align}
This average, for each fixed $t$, can be viewed actually as a  convolution  on the nilpotent group $\mathbb{G}_A  \cong \mathbb{R}^{d+1}$, where the group law on $ \mathbb{G}_A$ is defined as
$$(x,x_{d+1})\cdot (y,y_{d+1})=\left(x+y, x_{d+1}+y_{d+1}+ \langle Ax, y\rangle\right).$$ 
In 1997, Nevo and Thangavelu \cite{NT2}  initiated  the study of $\mathcal{M}_{S^{d-1}(A)}$ for a skew-symmetric matrix $A=J \in G_d(\mathbb{R})$ such that $J^t=-J$ with $d=2n$. In this case, the group $\mathbb{G}_J$ is the $(2n+1)$-dimensional Heisenberg group   $\mathbb{H}^n$, where the average (\ref{11}) is the group convolution of $f$ and the measure $\mu$ supported on the horizontal plane in the Heisenberg group 
\begin{align}\label{1212}
\{(y,0):y\in S^{d-1}\}\  \text{in}\ \mathbb{H}^n\cong \mathbb{R}^{2n+1}\ \text{where $d=2n$}.
\end{align}
In 2004,  Narayanan and Thangavelu in \cite{NT}, M\"{u}ller  and Seeger in \cite{MS} established:
\begin{align}\label{nsc112}
\| \mathcal{M}_{S^{d-1}(J)} \|_{L^p(\mathbb{G})\rightarrow L^p(\mathbb{G})} \leq C \quad \text{if and only if} \quad p > \frac{d}{d-1} 
\end{align}
where $\mathbb{G}=\mathbb{H}^{n}$ and $d=2n$ for $n \geq 2$.  Narayanan and Thangavelu employed the spectral theory of the Heisenberg group, while M\"{u}ller and Seeger recognized the average as a Fourier integral operator with phase having two-sided fold singularities as in \cite{GS2}. 
Subsequently, Anderson, Cladek, Pramanik, and Seeger in \cite{ACPS} treated the cases when the   measure $\mu$ is supported on the hyperplane in $\mathbb{H}^n$, distinct from the horizontal one:
\begin{align}\label{1313}
\{(y,\langle a,y\rangle):y\in S^{2n-1}\}\ \text{in}\ \mathbb{H}^n\ \text{with}\ a\in\mathbb{R}^{2n}.
\end{align}
 In \cite{MS}, M\"{u}ller  and Seeger  obtained  the maximal theorem (\ref{nsc112}) for  the averages  on the M\'etivier group  $\mathbb{G}_{\vec{J}}  \cong \mathbb{R}^{d+m}$ with $d=2n$ and $n\ge 2$, which is defined by $\vec{J}=(J_i)_{i=1}^m$ of $m$ skew symmetric $d\times d$ matrices $J_1,\cdots,J_m$: 
$$\mathcal{A}_{S^{d-1}(\vec{J})}f(x,\underline{x})=\int f\left(x-ty,  \underline{x}-\underline{y}- x^T\vec{J} y \right)d\sigma(y)\ \text{where}\  x^T\vec{J} y =\sum_{i=1}^m \left\langle  x, J_i y \right\rangle {\bf e}_i$$ 
under the non-degeneracy hypothesis that
 $\sum_{i=1}^m \theta_i J_i$  are invertible for all $\theta_i$.
Recently, in \cite{RS}, Ryu and Seeger obtained  (\ref{nsc112}) on  $\mathbb{G}_{\vec{J}} $ without the nondegeneracy assumption and the restriction $d=2n$. Their results cover  the all  step two nilpotent Lie groups $\mathbb{G}$  having natural local coordinates from Lie algebra $\mathfrak{g}$ via some exponential map.

The $L^p(\mathbb{H}^1)$ boundedness of $\mathcal{M}_{S^1(J)}$ for $p>2$  remains an open problem. However, Beltran, Guo, Hickman, and Seeger in \cite{BGHS} established the estimate $\| \mathcal{M}_{S^1(J)}f\|_{L^P(\mathbb{H}^1)} \leq C\|f\|_{L^p(\mathbb{H}^1)}$ for all $p>2$ in the case of radial functions $f \in L^p(\mathbb{H}^1)$ satisfying $f(x,x_3) = g(|x|,x_3)$.  In \cite{BHRT}, S. Bagchi, S. Hait, L. Roncal, and S. Thangavelu studied the lacunary version of $\mathcal{M}_{S^{d-1}(J)}$ and obtained $L^p$ boundedness results. They also investigated the $L^p(\mathbb{H}^n)\rightarrow L^q(\mathbb{H}^n)$ mapping properties of the localized version $\mathcal{M}^{loc}_{S^{d-1}(J)}$. In \cite{RSS}, Roos, Seeger, and Srivastava derived a sharp 
$L^p(\mathbb{H}^n)\rightarrow L^q(\mathbb{H}^n)$ estimate for $\mathcal{M}^{loc}_{S^{d-1}(J)}$ for $d=2n$ and $n\ge 2$.  More recently,  J. Lee and S. Lee  in \cite{LL}  proved again  the    aforementioned radial result of  \cite{BGHS}  in  a simple manner, and also  obtained the $L^p(\mathbb{H}^1)\rightarrow L^q(\mathbb{H}^1)$ improving estimates for the local radial maximal operators.  
For  lacunary maximal functions on   homogeneous groups and its broad applications, see the recent work of  Govindan,  Hickman, and Wright in \cite{GHW}.  

The main objective of this paper is to study the $L^p$ boundedness of $\mathcal{M}_{S^{d-1}(A)}$ for the general cases that
 $A$ are not given by the skew symmetric matrices.  In Section \ref{Sec23}, we set up our averages $\mathcal{A}_{S^{d-1}(A)}f(x,x_{d+1},1)$ as the convolution with the sphere measure on the horizontal plane lifted to the quadratic surfaces $(y, \langle (A+A^T)y/2, y\rangle/2) $ on the group induced by skew symmetric matrix $(A-A^T)/2$.

 \subsection{Results of Spherical Maximal Functions} 
  \begin{definition}
Given a matrix $A\in M_{d }(\mathbb{R})$, we define $\Lambda(A)$ as the set of all eigenvalues of $A$. Specifically, $\Lambda(A) =\{\mu\in \mathbb{C}:p_A(\mu)=0\}$, where $p_A(r)=\det(A-rI)$ is the characteristic polynomial of $A$. The exponent of the factor $(r-\mu)$ in $p_A(r)$ represents the algebraic multiplicity of the eigenvalue $\mu$, denoted as $m_a(\mu)$.
We denote the eigenspace $ \mathcal{N}(A-\mu I) $ corresponding to the eigenvalue $\mu$ as $\mathcal{E}(\mu)$, and its dimension as $m_g(\mu)$, which is called the geometric multiplicity.  
   \end{definition}
 \begin{main}\label{main1}
Let $A\in  M_{d}(\mathbb{R})$ with $d>2$. Suppose   $\Lambda(A)\cap \mathbb{R}=\emptyset.$ 
Then   
 \begin{align}\label{s81}
\left\| \mathcal{M}_{S^{d-1}(A)}\right\|_{L^p(\mathbb{R}^{d+1})\rightarrow L^p( \mathbb{R}^{d+1})}\le C\ \ \text{if and only if  $p>\frac{ d   }{ d- 1 }$.} 
\end{align}
\end{main}

We  use the notation $L^p(\mathbb{R}^{d+1})$ instead of $L^p(\mathbb{G}_A)$ without confusion.  
\begin{definition}
The nilpotent index of an eigenvalue $\mu\in \Lambda(A)$, denoted as $n(\mu)$, is defined as the smallest integer $m$ satisfying $$\mathcal{N}((A-\mu I)^{m})= \mathcal{N}((A-\mu I)^n)\ \text{for all $n\ge m$.}$$   
\end{definition}
In the case where $ \Lambda(A)\cap \mathbb{R}\ne \emptyset$, we will assume that for each $\mu\in\Lambda(A)\cap \mathbb{R}$, the nilpotent index $n(\mu)$ is equal to 1, which is equivalent to $m_a(\mu)=m_g(\mu)$. We denote the maximal multiplicity of such eigenvalues as $m_A$:
\begin{align}\label{m11}
m_A:=  
\max \{m_a(\mu):\mu\in \Lambda(A)\cap \mathbb{R}\}.
\end{align}  
Then we   state the $L^p$ results for some cases  of  $\Lambda(A)\cap \mathbb{R}\ne \emptyset$   as it follows.
\begin{main}\label{main2}
Let $A\in G_{d }(\mathbb{R})$ with $d>2$. Suppose that $\Lambda(A)\cap \mathbb{R}\ne\emptyset$ and the nilpotent indices $n(\mu)=1$ for all $\mu\in \Lambda(A)\cap \mathbb{R}$.
If $m_A <d-1$ in (\ref{m11}), then  
\begin{align*}
\left\| \mathcal{M}_{S^{d-1}(A)}\right\|_{L^p(\mathbb{R}^{d+1})\rightarrow L^p( \mathbb{R}^{d+1})}\le C\ \text{ if and only if $p>\frac{ d' }{ d'-1 }$  for  $d'= d- (m_A-1)$.}
\end{align*}
 If $m_A=d$,  then  $\| \mathcal{M}_{S^{d-1}(A)}\|_{L^p(\mathbb{R}^{d+1})\rightarrow L^p( \mathbb{R}^{d+1})}  =\infty$  for all $p<\infty$. 
\end{main}
\begin{remark}
If $\Lambda(A)\cap \mathbb{R}= \emptyset$, then take $m_A=1$ and 
Main Theorem 1  does not address the case $d=2$ concerning the circular maximal function  in   $L^p(\mathbb{H}^1)$ for   $p>2$.  Similarly,  Main Theorem 2 does not cover the case
$m_A =d-1$, which   corresponds to the  range $ p>2 $.
 \end{remark}
 \begin{remark}\label{ob11}
One can switch  $A$   with $A^T$ in the main theorems.
\end{remark}
 
\subsection{Sobolev Regularities on  $L^2(dxdt)$ and $L^2(dx)$} 
In the 1980s, Phong and Stein  \cite{PS} developed a powerful theory of singular and generalized Radon transforms, which relied on the observation of the smoothing effect arising from the rotational curvature of rotating planes  in the integral $\int_{\mathbb{R}^d} f(x+y, x_{d+1}+\langle Ax, y\rangle) dy$. The exact amount of smoothness is $d/2$ Sobolev derivatives in $L^2(\mathbb{R}^{d+1})$, gained from   the variable planes $(x,x_{d+1})\cdot (\mathbb{R}^d\times\{0\})$   in  the nilpotent group $\mathbb{G}_A$.   
As a co-dimension 2 analogue,   the average $\mathcal{A}_{S^{d-1}(A)}f(x,x_{d+1},t)$ is  taken over the variable spheres $(x,x_{d+1})\cdot (S^{d-1}_t\times\{0\})$ in the nilpotent group $\mathbb{G}_A \cong \mathbb{R}^{d+1}$.  If $A$ is skew-symmetric matrix $J$,    then the average $\mathcal{A}_{S^{d-1}(A)}f(\cdot,t)$ preserves  the amount of $(d-1)/2-1/6$ derivatives in $L^2(dxdx_{d+1})$ for a fixed $t$, wherease it preserves the $(d-1)/2$   derivatives in $L^2(dxdx_{d+1}dt)$.  This is a key feaure for deriving the previous results  in  the Heisenberg group,   which is not possible
  in  the Euclidean $L^2$ space.  
In this paper, we explicitly assert that $\mathcal{A}_{S^{d-1}(A)}f(\cdot,t)$ has at least 1/6 more $L^2$-smoothness gain in the nilpotent group $\mathbb{G}_A$ from the $dt$ integral if either (i) $\Lambda(A)\cap \mathbb{R}=\emptyset$ or (ii) $\Lambda(A)\cap \mathbb{R}\ne \emptyset$ with $n(\mu)=1$ for all $\mu\in \Lambda(A)\cap \mathbb{R}$.

 \begin{main}\label{main4}[Sobolev Regularity]
 Let $ A\in G_{d }(\mathbb{R})$ for $d\ge 2$ and $t\ne 0$.  
Suppose  that  (i) $\Lambda(A)\cap \mathbb{R}=\emptyset$ or (ii)   $\Lambda(A)\cap \mathbb{R}\ne \emptyset$ and $n(\mu)=1$ for every $\mu\in \Lambda(A)\cap \mathbb{R}$. Then  we have
\begin{align}\label{0045}
& \|\mathcal{A}_{S^{d-1}(A)}\|_{L^2(\mathbb{R}^{d+1})\rightarrow L^2_{\alpha}(\mathbb{R}^{d+1}\times [1,2])  }\lesssim 1\ \text{if $\alpha< \frac{d-1}{2} -\frac{m_A-1}{2}$,} \\
& \|\mathcal{A}^t_{S^{d-1}(A)}\|_{L^2(\mathbb{R}^{d+1})\rightarrow L^2_{\alpha}(\mathbb{R}^{d+1})  } =  \infty\ \text{ if  $\alpha> \frac{d-1}{2}-\frac{m_A-1}{2}-\frac{1}{6}$,}\label{a245}
\end{align}
where  in case of  $A^T=-A$ it holds that for fixed $t\ne 0$,
  $$\|\mathcal{A}^t_{S^{d-1}(A)}\|_{L^2(\mathbb{R}^{d+1})\rightarrow L^2_{\alpha}(\mathbb{R}^{d+1})  } \lesssim 1\ \text{ if $\alpha< \frac{d-1}{2}-\frac{1}{6}$.}$$
\end{main} 
 
\begin{remark}
The exact ranges of $\alpha$ in (\ref{a245}), where $A^T\ne -A$ are currently unknown, and our argument does not include the endpoints of the $\alpha$ ranges in (\ref{0045}) and (\ref{a245}).
For the case $p\ne 2$, we refer Pramanik and Seeger \cite{PS} and Bentsen \cite{B} for the $L^p\rightarrow L^p_{\alpha}$ regularity of     average operators $\mathcal{A}^t$ for fixed $t$   associated with   curves   in the Heisenberg group $\mathbb{H}^1$, in the general context of the fold singularities. These results are preceded by  the    model curves $\{(s,s^2, c s^3)\}$ with $c\ne 1/6$ in $\mathbb{H}^1$ studied by Secco \cite{S}. However,  the boundedness of $\|\mathcal{A}_{S^{1}(A)}\|_{L^p(\mathbb{H}^{1})\rightarrow L^p_{\alpha}(\mathbb{H}^{1}\times [1,2])  }$ is not known for any $p>2$ and $\alpha>1/2$. \end{remark}
\begin{remark}
The initial version of this paper is the  preprint \cite{K3}. It includes
the  result of the $L^2$ boundedness of the   maximal operator  
 $\mathcal{M}_{S^{d-1}_{\delta}(A)}$ 
associated with the annuli
$S^{d-1}_\delta=S^{d-1}+O(\delta)$. 
Suppose that  (i) $\Lambda(A)\cap \mathbb{R}=\emptyset$ or (ii)    $n(\mu)=1$ for all $\mu\in \Lambda(A)\cap \mathbb{R}$.  Then,  from Main Theorems 1-3, we can state 
\begin{align}\label{004}
\|\mathcal{M}_{S^{d-1}_{\delta}(A)}\|_{L^2(\mathbb{R}^{d+1})\rightarrow L^2(\mathbb{R}^{d+1})}\approx_{\epsilon}    \delta ^{(d-2)/2 -(m_A-1)/2}
\end{align}
where the exponent $\frac{d-2}{2}-\frac{m_A-1}{2}$ in $L^{2,\infty}(dxdx_{d+1} dt)$ in (\ref{004})  matches with
the endpoint $\alpha_0=\frac{d-1}{2}-\frac{m_A-1}{2}$ in $L^2_{\alpha_0}(dxdx_{d+1}dt)$ in (\ref{0045}).
If $d=2$, we can include the case $n(\mu)=2$ and
classify
 $\|\mathcal{M}_{S^{1}_{\delta}(A)}\|_{L^2(\mathbb{R}^{3})\rightarrow L^2(\mathbb{R}^{3})}\approx
\delta^{-1/2} \delta^{\alpha_0} $ 
where $\alpha_0=\min\{\frac{1}{2}-\frac{m_a(\mu)-1}{2}+\frac{n(\mu)-1}{3}:\mu\in \Lambda(A)\cap\mathbb{R}\} $. Moreover, $\alpha_0$  is the endpoint of (\ref{0045}) for $d=2$. However, a comprehensive classification, including all $n(\mu)\ge 2$, remains unresolved in higher dimensions ($d \geq 3$).  
\end{remark}

  \section{Step Two Nilpotent Group}
To each $A=(a_{ij})\in M_{d}(\mathbb{R}) $, we   assign a  Lie group $\mathbb{G}_A$  identified with $\mathbb{R}^{d+1}$ endowed with the group multiplication:
$$(x,x_{d+1})\cdot (y,y_{d+1})=\left(x+y,x_{d+1}+y_{d+1}+\langle Ax,y\rangle\right).$$
We are treating  with the spherical means 
 $\mathcal{A}_{S^{d-1}(A)}(f)(x,x_{d+1})$ 
as an actual   convolution  in   $\mathbb{G}_A$  of $f$ and the measure $\mu$ with   $\mu(f)=\int_{y\in S^{d-1}} f(y,0)d\sigma(y) $.
We  take the corresponding Lie algebra $\mathfrak{g}_A$ as the tangent space at the identity ${\bf 0}$ of $\mathbb{G}_A$ consisting of   left invariant vector fields $X^{v} $ with respect to left translate in $ \mathbb{G}_A$:
\begin{align}\label{hh}
X^{v}(f)(x,x_{d+1})= \frac{d}{dt} \left( f((x,x_{d+1})\cdot (\gamma(t))\right)|_{t=0} 
\end{align}
 where $t\rightarrow\gamma(t)$ is a curve in $\mathbb{G}_A$ with $\gamma(0)=0$ and $\gamma'(0)=v\in \mathbb{G}_A$.   Then   $\mathfrak{g}_A$  generated by the basis of    $\{X_j  \}_{j=1}^{d+1}$ where $X_j$ is $X^{v} $ with $v={\bf e}_j$: 
$$X_j=\partial/\partial x_j+(a_{j1}x_1+\cdots +a_{jd}x_d)\partial/\partial_{x_{d+1}}\ \text{for $j\in [d]$ and}\ X_{d+1}=\partial/\partial x_{d+1}$$
  forms a vector space. Moreover,
 they satisfy the commutator relations $$[X_i,X_j]=(a_{ji}-a_{ij})X_{d+1}\ \text{ for $ i,j\in [d]$,}$$ with the other commutators vanished. Thus for any matrix $A$, it holds that  $$[[ \mathfrak{g}_A,\mathfrak{g}_A], \mathfrak{g}_A]=0,\ \text{ i.e.,  the step of   $\mathfrak{g}_A$ is   2 when $A^T\ne A$}.$$

\subsection{Convolution Form in $\mathbb{G}_A$}
The group $\mathbb{G}_A$ has the inverse element of $(y,y_{d+1})$:
 $$(y,y_{d+1})_A^{-1}=\left(-y,-y_{d+1}+\langle Ay,y\rangle\right).$$ 
Thus the group convolution of $f*_{\mathbb{G}_A}g$ for two   functions $f$ and $g$  is  given by
\begin{align}\label{con1}
f*_{\mathbb{G}_A}g(x,x_{d+1})&=\int f\left((x,x_{d+1})\cdot (y,y_{d+1})_A^{-1}\right) g(y,y_{d+1})dydy_{d+1}\nonumber\\
&= \int f(x-y,x_{d+1}-y_{d+1}-\langle A(x-y),y\rangle)g(y,y_{d+1})dydy_{d+1}
\end{align}
where $dydy_{d+1}$ is the Haar measure of $\mathbb{G}_A$. For a simpler expression, we set 
\begin{align*}
f \circ_A g(x,x_{d+1})&:= \int f(x-y,x_{d+1}-y_{d+1}-\langle Ax,y\rangle)g(y,y_{d+1})dydy_{d+1}\ \text{so that}\\
 f\circ_A g&=\mathcal{U}_{-A}(\mathcal{U}_{A}f*_{\mathbb{G}_{-A^T}}  g)   
\ \text{ with $\mathcal{U}_Af(x,x_{d+1}):=f(x,x_{d+1}+\langle Ax,x\rangle)$. }
\end{align*} 
Then one can  represent  the average $\mathcal{A}_{S^{d-1}(A)}f(x,x_{d+1},t)$  in  (\ref{11}) as
\begin{align}\label{isig}
 f\circ_A  [d\sigma_t\otimes \delta_{d+1}] (x,x_{d+1})  
& =\mathcal{U}_{-A}(\mathcal{U}_Af*_{\mathbb{G}_{-A^T}}  [d\sigma_t\otimes \delta_{d+1}]  )(x,x_{d+1}).  
\end{align}
By this with Remark \ref{ob11},   we regard   $\mathcal{A}_{S^{d-1}(A)}f(\cdot,t)$ as   $f *_{\mathbb{G}_{A}}[d\sigma_t\otimes \delta_{d+1}](\cdot)$ in $\mathbb{G}_A$.

\subsection{Failure of Local Coordinates of $\mathbb{G}_A$ from   $\mathfrak{g}_A$ when $A\ne A_{\rm{skew}}$}
To identify the Lie algebra $\mathfrak{g}_A$ and the Lie group $\mathbb{G}_A$ via $\text{exp}:\mathfrak{g}\rightarrow\mathbb{G}_A$  preserving the group law in the local coordinates, we may assume that there is  a one-parameter family $\left\{\text{exp}(t X)\right\}_{t} $ of diffeomorphisms  for small $t$ satisfying  $$ \text{exp}\left(\sum_{j=1}^{d+1} u_j X_j  \right)(x,x_{d+1})=(u,u_{d+1}) \cdot (x,x_{d+1}), 
$$
which gives the local coordinates $u$ for the group $\mathbb{G}_A$.   However, some issues arise from the Campbell-Hausdorff formula,  $$\text{exp}(X)\cdot \text{exp}(Y)=\text{exp}(X+Y+\frac{1}{2}[X,Y]).$$ For $X=\sum u_i X_i$ and $Y=\sum v_j X_j$  with $ A_{\rm{skew}}=\frac{ (A-A^T)}{2}$, 
$$X+Y+\frac{1}{2}[X,Y] =  \sum_{i=1}^d(u_i+v_i) X_i+\left( (u_{d+1}+v_{d+1})+ \langle A_{\rm{skew}} u,v\rangle \right)X_{d+1}.$$
 Thus, the coordinate of the RHS of the Campbell-Hausdorff formula is 
\begin{align}\label{iia}
&\text{exp}\left(X+Y+\frac{1}{2}[X,Y]\right)(x,x_{d+1})\nonumber\\
 &\qquad= \left( u+v, u_{d+1}+v_{d+1}+ \langle A_{\rm{skew}} u,v\rangle  \right) \cdot (x,x_{d+1})
\end{align}
removing the  symmetric part $A_{\rm{sym}}= \frac{ (A+A^T)}{2}$. Meanwhile the LHS   is
\begin{align*}
\text{exp}(X)\cdot \text{exp}(Y)(x,x_{d+1}) =(u,u_{d+1})\cdot (v,v_{d+1})\cdot (x,x_{d+1}) .
\end{align*}
Thus, in  cases $A\ne A_{\rm{skew}}$,  the algebra  $\mathfrak{g}_A$ with the exponential map  would not give a local coordinate of the Lie group $\mathbb{G}_A$ that  preserves the group law.
\subsection{Interpretation of $\mathcal{A}_{S^{d-1}(A)}$ in the Context of the Lie Theory}  \label{Sec23}
 To   keep this  Lie theory, we  can switch our base group $\mathbb{G}_A$ with the group $\mathbb{G}_{A_{\rm{skew}}}$ induced from the skew symmetric matrix $A_{\rm{skew}}= \frac{ (A-A^T)}{2}$, having the local coordinate from the Lie algebra $\mathfrak{g}_{A_{\rm{skew}}}$ via $\text{exp}$. By utilizing the identity
\begin{align*}
\left\langle Ax,y\right\rangle&=\frac{1}{2}\left\langle A_{\rm{sym}}(x-y),(x-y)\right\rangle
-\left\langle A_{\rm{skew}}x,y\right\rangle -\frac{1}{2}\left\langle A_{\rm{sym}} x,x\right\rangle-\frac{1}{2}\left\langle A_{\rm{sym}} y,y\right\rangle
\end{align*}
we can express  $\mathcal{A}_{S^{d-1}(A)}f$ as the convolution in the Heisenberg type group $\mathbb{G}_{A_{\rm{skew}}}$:
\begin{align}\label{03}
[\mathcal{A}_{S^{d-1}(A)}f](x,x_{d+1})&= [f]*_{\mathbb{G}_{A_{\rm{skew}}}}\mu_{\rm{sym}} \left(x,x_{d+1}\right)
\end{align}
where $[f](x,x_{d+1})=f(x,x_{d+1}+ \frac{1}{2}\langle A_{\rm{sym}}x,  x\rangle)$ and
  the measure $\mu_{\rm{sym}}$ is
$$
 \mu_{\rm{sym}}(f) =\int_{y\in S^{d-1}} f\left(y, \frac{1}{2}\left\langle A_{\rm{sym}}y,y\right\rangle\right)d\sigma(y).
$$
Therefore, from  the RHS of  (\ref{03}) expressed as $$ [f]*_{\mathbb{G}_{A_{\rm{skew}}}}\mu_{\rm{sym}} \left(x,x_{d+1}\right)=\int_{S^{d-1}} [f]\left(x-y, x_{d+1}-\langle A_{\rm{skew}} x,y\rangle -\frac{1}{2}\langle A_{\rm{sym}} y,y\rangle\right) d\sigma(y)$$
we can interpret $ \mu_{\rm{sym}}$ as the spherical measure, lifted to the quadratic surfaces  in  $\mathbb{G}_{A_{\rm{skew}}}$, supported $\{(y, \frac{1}{2}\left\langle A_{\rm{sym}}y,y\right\rangle):y\in S^{d-1}\} $. Compare it with (i) and (ii):
\begin{itemize}
\item[(i)]
the sphere  measure   on the horizontal plane $\{(y,0):y\in S^{d-1}\}$  in (\ref{1212})
\item[(ii)]
  the sphere measure lifted to the hyper-plane $\{(y,\langle a,y\rangle),y\in S^{d-1}\}$ in (\ref{1313}).
\end{itemize}

From the  results of many authors, it can be observed that the skew symmetric part  $\langle A_{\rm{skew}}x,y\rangle$,    without the symmetric part $\langle A_{\rm{sym}}y,y\rangle$, keeps the curvature of the sphere $S^{d-1}$ in $\mathbb{R}^d$, maintaining the $L^p$ boundedness properties of the maximal averages on the Heisenberg type groups.   However, in this paper, we will demonstrate that the symmetric part $\langle A_{\rm{sym}}y,y\rangle $ tends to eliminate some curvatures of the sphere in many cases of $A$. For example, if $A=I$ such that $A_{\rm{sym}}=I$ and $A_{\rm{skew}}={\bf 0}$, then
$\mathcal{M}_{S^{d-1}(A)}$ is not bounded in any $L^p(\mathbb{R}^{d+1})$. 
 The purpose of this paper is to investigate the extent to which  the symmetric part $\frac{1}{2}\langle A_{\rm{sym}} y,y\rangle$ interferes with   curvature of $S^{d-1}$ in  the group $\mathbb{G}_{A_{\rm{skew}}}$. 
\\

\noindent
{\bf Notations}.  
 Given  two scalar expressions $A,B$,  we write $A\lesssim B$   if $A\le CB$ for a   constant $C>0$ independent of $A,B$. The notation $A\approx B$ indicates $A\lesssim B$ and $B\lesssim A$. We write  $A\lesssim   2^{\epsilon j}B$ if there is $C_\epsilon$ such that $A\le C_\epsilon 2^{\epsilon j}B$ for sufficiently small $\epsilon$. Our positive constants $c$ and $C$ can be also different   line by line.  
   Given two vectors $u$ and $v$ in $\mathbb{R}^d$, we write 
$v=u+O(\rho)$ if there exists $C>0$ independent of $u$, $v$, and $\rho$ such that
\begin{align}\label{30kg}
|v-u|\le C\rho.
\end{align}
In this paper,  we employ the following smooth non-negative cutoff functions
 \begin{itemize}
 \item[(1)] $\psi$ supported in $\{u\in\mathbb{R}^d :|u|\le 1\}$ for $\psi(u)\equiv 1$ in $|u|<\frac{1}{2}$ and $\psi^c=1-\psi$,
 \item[(2)]  $\chi$  supported in $\{ u\in \mathbb{R}^m:1/2\le |u|\le 2\}   $  for $m=1,d$ depending on cases,
  \end{itemize}
  allowing slight line-by-line modifications of $\chi$ and $\psi$.    We denote the phase functions  by $\Phi,\phi$, and  the integral kernels by $K,\Psi$, which can be different in cases.   
Write $[d]=\{1,\cdots,d\}$. Finally, for $A,B\in M_d(\mathbb{R})$, we denote $A\sim B$ if $A$ is row equivalent to $B$, and  we write $\mathcal{C}(A)$ as the column space of   $A$ and $\mathcal{P}_S$  as the projection matrix from $\mathbb{R}^d$  to a subspace $S\subset \mathbb{R}^d$.

 \subsection{Idea of Proof}
In \cite{MS}, the average operator $\mathcal{A}_{S^{d-1}(J)}$ is treated as the generalized Radon transform and Fourier integral operator.    Their estimate is based on the observation that the corresponding phase functions have two-sided fold singularities.
In this paper,  we express $\mathcal{A}_{S^{d-1}(A)}$ as a pseudodifferential operator with the symbol $\widehat{d\sigma}\left(t(\xi+ \xi_{d+1} Ax)\right)$ so that
 $\mathcal{A}_{S^{d-1}(A)}f(x,x_{d+1},t)$ in (\ref{11}) is  
\begin{align}\label{49b}
&\mathcal{A}_{S^{d-1}(A)}f(x,x_{d+1},t)\nonumber\\
&\qquad=\int_{y\in S^{d-1}}\left(\int_{\mathbb{R}^{d+1}} e^{2\pi i (\xi,\xi_{d+1})\cdot (x-ty, x_{d+1}-\langle Ax,ty\rangle)} \widehat{f}(\xi,\xi_{d+1})d\xi d\xi_{d+1} \right)d\sigma(y)\nonumber\\
&\qquad =  \int_{\mathbb{R}^{d+1}} e^{2\pi i \langle (x,x_{d+1}),(\xi,\xi_{d+1})\rangle}\widehat{d\sigma}\left(t(\xi+ \xi_{d+1} Ax)\right)\widehat{f}(\xi,\xi_{d+1})d\xi d\xi_{d+1}. 
\end{align}
On the second line above, we  used  the Fourier inversion formula in the Euclidean space $\mathbb{R}^{d+1}$. Here  $\widehat{d\sigma}$ denotes the Fourier transform of the measure $d\sigma$, defined in $\mathbb{R}^d$,  given by
\begin{align}
\widehat{d\sigma}(\xi)= 2\pi J_{(d-2)/2}(2\pi |\xi|)|\xi|^{-\left( \frac{d-2}{2}\right)}\ \text{for $\xi\in \mathbb{R}^d$}. \label{p54}
\end{align}
So  the asymptotic expansion of  the Bessel function $J_{(d-2)/2}$ enables us
  to utilize the     wave propagating factor $e^{2\pi it|\xi|}$. In our proof,  the main case  is $|\xi+\xi_{d+1}Ax|\approx |\xi_{d+1}|\approx 2^j$ in (\ref{49b}), which corresponds to the    oscillatory integral operators $\mathcal{T}_j$:
 \begin{align*}
 \mathcal{T}_j g(x,t)&=  2^{(d/2)j}  \chi\left(t\right)\psi\left(x\right) \int_{\mathbb{R}^d}
e^{2\pi i 2^j\phi(x,t,\xi)}
\chi\left(  t|\xi+x| \right) \widehat{g}(  \xi)d\xi\ \text{where}\\
 \phi(x,t,\xi)&=\langle Ax, \xi\rangle+t|\xi+x|.
\end{align*}
The mixed hessian of the phase $\phi$ with respect to $(x,t)$ and $\xi$ is $(d+1)\times d$ matrix: 
\begin{align}
\phi_{(x,t)\xi}''=\left(\begin{matrix}\phi''_{x\xi}\\
\phi''_{t\xi}\end{matrix}\right)=\left(\begin{matrix}\left(A+\frac{t }{|u|}I\right)- \frac{t (u_iu)}{|u|^3} \\
u/|u|
 \end{matrix}\right) \sim  \left(\begin{matrix}\left(A+\frac{t }{|u|}I\right)  \\
u/|u|
 \end{matrix}\right) \label{jy}
 \end{align}
where $(u_iu)_{i=1}^d\in M_d(\mathbb{R})$ with $u=x+\xi$. 
We obtain Main Theorem 1    from $\|\mathcal{T}_j\|_{op}=O(1)$ due to  H\"{o}rmander's decay  as
\begin{itemize}
\item the condition $\Lambda(A)\cap \mathbb{R}=\emptyset$ yields  the non-degeneracy of   $A+\frac{t }{|u|}I$ in (\ref{jy}).
\end{itemize}
For   Main Theorem 2,  we split  $A+\frac{t }{|u|}I$ in $\phi''_{x\xi}$  as   $ (A-\mu I)+(\mu+\frac{t}{|u|}) I$ for $\mu\in \Lambda(A)\cap \mathbb{R}$ in (\ref{jy}). Our key estimate is
   $$\|\mathcal{T}_j\|_{op}=O( 2^{(d/2)j}   2^{-j(d-m_A)/2}2^{-j/2} ) =O(2^{(m_A-1)j}) $$
which is $O(1)$ if $m_A=1$.  This estimate
 is based on the observation.
\begin{itemize}
\item the main part $ A-\mu I $ in $\phi''_{x\xi}$ behaves like  the projection $\mathcal{P}_{\mathcal{E}^{\perp}(\mu)} $ due to $\mathcal{C}(A-\mu I)\oplus \mathcal{E}(\mu)=\mathbb{R}^d$ under the hypothesis $n(\mu)=1$.  This derives the decay   $2^{-j(d-m_A)/2} $, and   $u/|u|$ in  $\Phi''_{t\xi}$ generates
 the other   $2^{-j/2}$
in $\|\mathcal{T}_j\|_{op}$. 
\end{itemize} 
For (\ref{a245}) of Main Theorem 3,   the matrix $\frac{t (u_iu)}{|u|^3}  $ in the middle of (\ref{jy}) is not removed due to the absence of  $\phi_{t\xi}''$. However,  $(u_iu)\sim |u|({\bf e}_1|{\bf 0}|\cdots|{\bf 0})$  helps to 
compute  $\det(\phi''_{x\xi})$. Note  that
 it is  $|\xi+\xi_{d+1}Ax|\ll |\xi_{d+1}|$   in (\ref{49b}) where we lose    $1/6$ derivative  when $A$ is  skew symmetric.
\\
\\
{\bf Organization}
In Section 2, we replace the estimate of the maximal functions with that of the oscillatory integral operators $\mathcal{T}_j$.
In Section 3, we establish basic oscillatory estimates and prove Main Theorem 1.
In Section 4, we utilize the linear algebra $\mathcal{C}(A-\mu I)\oplus \mathcal{E}(\mu)=\mathbb{R}^d$ to prove Main Theorem 2.
In Section 5, we obtain the necessity of Main Theorems 1 and 2.
In Section 6, we prove the regularity estimate for Main Theorem 3.
In the Appendix, we provide a rather standard argument for proving a bulk of lemmas on the nilpotent groups.
\section{Reduction to Oscillatory Integral Operators $\mathcal{T}^\lambda$}\label{sec3}
To change  the maximal estimate to that of the  oscillatory integral operators $\mathcal{T}_j$, we shall    utilize  (i) the asymptotic expansion of $\widehat{d\sigma}(\xi)$,  (ii) the majorization of $L^\infty(dt)$ by $L^2_{1/2}(dt)$ and (iii)  the Plancherel theorem with respect to  $x_{d+1}$. 
\subsection{Frequency Decomposition}
Let  $\widehat{d\sigma}$ be the Euclidean Fourier transform of the
 measure $d\sigma$ on the unit sphere $S^{d-1}$. Decompose
 $  \widehat{d\sigma}(\xi)=\sum_{j=0}^\infty\widehat{d\sigma_j}(\xi)$ 
 with  
\begin{align}\label{kc1}
\widehat{d\sigma_0}(\xi)&=\widehat{d\sigma}(\xi)\psi(\xi)
  \ \text{and}\ \widehat{d\sigma_j}(\xi)=\widehat{d\sigma}(\xi)\chi\left(\frac{\xi}{2^j}\right)\ \text{for $j\ge 1$}
\end{align}
 having the frequency support $|\xi|\le 1$ and $|\xi|\approx 2^j$ respectively. Next define
\begin{align}
\mathcal{A}_{j} f(x,x_{d+1},t)&: 
 = \int_{\mathbb{R}^d}  f\left(x-ty,x_{d+1}-\langle Ax, ty\rangle\right)d\sigma_j(y) \label{kcc1}
  \end{align} 
where $(x,x_{d+1},t)\in\mathbb{R}^{d+1}\times\mathbb{R}$. To each   integer $j \ge 0$, we assign the corresponding maximal operator $\mathcal{M}_j$ defined by
\begin{align*} 
\mathcal{M}_jf(x,x_{d+1})= \sup_{t>0}|\mathcal{A}_{j}f(x,x_{d+1},t)|.
\end{align*}
Using the well-known estimate for 
$j\ge 0$ in (\ref{kc1}), we have
\begin{align*}
 \int e^{2\pi i\xi\cdot x}\widehat{d\sigma_j}(\xi)d\xi=\int_{S^{d-1}} 2^{dj} \chi^{\vee}\left(2^j(x-y)\right)d\sigma(y)=O\left(\frac{2^j}{(1+|x|)^{N}}\right).
\end{align*}
From this, we  deduce the existence of a constant 
 $C_p$ for every $1<p\le \infty$ such that
\begin{align}
 \|\mathcal{M}_j\|_{L^{p}(\mathbb{R}^{d+1})\rightarrow L^p(\mathbb{R}^{d+1})}\le  C_p2^j.\label{kcc2}
\end{align}  
A detailed proof of this estimate will be provided in the Appendix.

\subsection{Symbol Representations}
In view of  (\ref{49b}),  express $\mathcal{A}_{j} f$ in      (\ref{kcc1}) as
\begin{align*}
\mathcal{A}_{j} f(x,x_{d+1},t)
& =\int_{\mathbb{R}^{d+1}} e^{2\pi i \langle (x,x_{d+1}),(\xi,\xi_{d+1})\rangle}\widehat{d\sigma_{j}}\left(t(\xi+ \xi_{d+1} Ax)\right)\widehat{f}(\xi,\xi_{d+1})d\xi d\xi_{d+1}.
\end{align*}
For $t>0$  dilate the measure $(d\sigma_j)_t$ such that $\widehat{(d\sigma_j)_t }(\xi)=\widehat{(d\sigma_j)}(t\xi)$.
  Next for each $k\in\mathbb{Z}$, 
from (\ref{isig}),   set  the average:
\begin{align}\label{0p}
\mathcal{A}_{j,k} f(x,x_{d+1},t):=f\circ_A [(d\sigma_j)_t\otimes \delta](x,x_{d+1}) \chi(2^kt) 
\end{align}
whose   symbol 
 is $ \widehat{d\sigma_{j}}\left(t\big(\xi+ \xi_{d+1} Ax\big)\right) \chi(2^kt)$ in the form of pseudo differential operator.
 Then the following pointwise inequality holds,
 \begin{align} \label{115a}
|\mathcal{M}_{j}f(x,x_{{d+1}})|&\le\sup_{k\in\mathbb{Z}}\sup_{t>0} \left| \mathcal{A}_{j,k}f(x,x_{{d+1}},t)\right|.  \end{align} 
 
        \subsection{Majorizing $L^\infty(\mathbb{R}_+)$ by $L^2_{1/2}(\mathbb{R}_+)$}
In   (\ref{0p}),  it holds that
$$ \partial_t\mathcal{A}_{j,k}f(x,x_{d+1},t)=   
f\circ_A   [ \partial_t[(d \sigma_{j})_t\chi(2^kt)]\otimes \delta_{d+1}](x,x_{d+1}).$$
Here,   the frequency part of $   \partial_t[(d \sigma_{j})_t\chi(2^kt)] $ at $\xi\in \mathbb{R}^d$ can be expressed as
\begin{align}\label{a45}
  2^k\left[ \widehat{d\sigma}(t\xi) \tilde{\chi}\left(\frac{t|\xi|}{2^j}\right) \tilde{\chi}(2^kt) \right]+ 2^{j+k}\left[\frac{\xi}{|\xi|}\cdot\nabla[\widehat{d\sigma}](t\xi)\chi\left(\frac{t|\xi|}{2^j}\right)\chi(2^kt)\right].
\end{align}
On the other hand,  switch  $\sup_t$ of  (\ref{115a}) with $L^2(dt)$ in below,  
\begin{align}\label{p52}
\sup_{t>0} |\mathcal{A}_{j,k}f(x,x_{{d+1}},t)|^{2}&  \le 2\int_{0}^\infty \left|\partial_t \left(\mathcal{A}_{j,k}f(x,x_{{d+1}},t)  \right) \right|  \left|\mathcal{A}_{j,k}f(x,x_{{d+1}},t) \right| dt\nonumber\\
&\le C  2^{j+k} \int_{0}^\infty \left| \tilde{\mathcal{A}}_{j,k}f(x,x_{{d+1}},t)    \right|^2 dt.
\end{align}
We take  $\tilde{\mathcal{A}}_{j,k}$ as the pseudodifferential operator defined by 
$$\tilde{\mathcal{A}}_{j,k}f(x,x_{d+1},t)=   
f\circ_A   [ (d\tilde{\sigma}_{j})_t\otimes \delta_{d+1}](x,x_{d+1})\chi(2^kt), $$
with  the symbol  $ \widehat{(d\tilde{\sigma}_{j})_t} \left(\big(\xi+ \xi_{d+1} Ax\big)\right) \chi(2^kt)$, expressed as
\begin{align}
\tilde{\mathcal{A}}_{j,k}f(x,x_{d+1},t)&=\int e^{2\pi i(x,x_{d+1})\cdot (\xi,\xi_{d+1})}\widehat{d \tilde{\sigma}} \left(t\big(\xi+ \xi_{d+1} Ax\big)\right) \nonumber\\
&\times\chi\left( \frac{t\big(\xi+ \xi_{d+1} Ax\big)}{2^j} \right)\chi(2^kt)\widehat{f}(\xi,\xi_{d+1})d\xi d\xi_{d+1}. \label{3990}
\end{align}
Here  we choose $\widehat{d\tilde{\sigma}}$   from  the second term of (\ref{a45})  since it dominates the first term.
\begin{align}\label{i7}
\widehat{d\tilde{\sigma}}(\xi)&=\int_{S^{d-1}}  e^{2\pi i  \xi \cdot y} \,  (\xi/|\xi|)\cdot y   d\sigma(y)+\widehat{d\sigma}(\xi).
\end{align}
Since  the decay rates of $\int_{S^{d-1}}  e^{2\pi i  \xi \cdot z} \,  (\xi/|\xi|)\cdot z   d\sigma(z) $  and $\widehat{d\sigma}(\xi)$ in $|\xi|$ are similar, we can identify  $\widehat{d\tilde{\sigma}}(\xi)$ with 
$\widehat{d\sigma}(\xi)$ in (\ref{3990}). By (\ref{115a}) and (\ref{p52}), we have
      \begin{align}
\|\mathcal{M}_jf\|_{L^2(\mathbb{R}^{d+1})}  &\le  2^{j/2} \left(\sum_{k\in\mathbb{Z}}\left\|2^{k/2}\tilde{\mathcal{A}}_{j,k}f\right\|_{L^2(\mathbb{R}^{d+1}\times \mathbb{R})}^2\right)^{1/2}. \label{oj11}
\end{align}  
In the next two subsection, we   reduce our matters to the case   $k=0$ in (\ref{oj11}).

\subsection{Littlewood-Paley Decomposition}
Note that $\psi$  is supported in $|(\xi,\xi_{d+1})|\le1$ in $\mathbb{R}^{d+1}$ and
 $\psi(\xi,\xi_{d+1})\equiv 1$ on $|(\xi,\xi_{d+1})|\le1/2$. Using the nonisotropic dilation, we set
\begin{align}\label{10061}
\widehat{P_j}(\xi,\xi_{d+1})=\psi\left( \frac{\xi}{2^{j+1}},\frac{\xi_{d+1}}{2^{2(j+1)}}\right)- \psi\left( \frac{\xi}{2^{j}},\frac{\xi_{d+1}}{2^{2j}}\right),
\end{align}
where $\widehat{}$ is the Fourier transform in the Euclidean space $\mathbb{R}^{d+1}$.
We consider the Littlewood-Paley projection $\mathcal{P}_jf$   of $f\in \mathcal{S}(\mathbb{R}^{d+1})$ by defining $\mathcal{P}_jf(y,y_{{d+1}})$ as  
\begin{align*}
& \int_{  (\eta,\eta_{d+1})\in \mathbb{R}^d\times \mathbb{R}}e^{2\pi i \langle (\eta,\eta_{{d+1}}),(y,y_{{d+1}}) \rangle}\widehat{P_j}(\eta+\eta_{d+1}Ay,\eta_{d+1})\widehat{f}(\eta,\eta_{{d+1}})d\eta d\eta_{{d+1}}. 
\end{align*}
By using $\sum_{\ell\in\mathbb{Z}} \mathcal{P}_{j+k+\ell}=Id$ in (\ref{oj11}), one has
\begin{align}
&RHS\ \text{of (\ref{oj11})}  \lesssim 2^{j/2} \sum_{\ell\in\mathbb{Z}} \left\| \left(\sum_{k\in \mathbb{Z}}|2^{k/2}\tilde{\mathcal{A}}_{j,k}\mathcal{P}_{j+k+\ell} f|^2\right)^{1/2}\right\|_{L^2(\mathbb{R}^{{d+1}}\times \mathbb{R})}.\label{0453}
\end{align} 
\begin{proposition}\label{prop03} 
Consider   $\tilde{\mathcal{A}}_{j,k}$ in (\ref{3990}). For an arbitrary small $\epsilon>0$, it holds,
\begin{align*}
& \sum_{\ell\in\mathbb{Z}} \left\| \left(\sum_{k\in \mathbb{Z}}|2^{k/2}\tilde{\mathcal{A}}_{j,k}\mathcal{P}_{j+k+\ell} f|^2\right)^{1/2}\right\|_{L^2(\mathbb{R}^{{d+1}}\times \mathbb{R})}\\
&\qquad\qquad\qquad \lesssim  2^{\epsilon j} \left(1+\sup_k  \left\| 2^{k/2}\tilde{\mathcal{A}}_{j,k}  \right\|_{ L^2(\mathbb{R}^{d+1})\rightarrow L^2(\mathbb{R}^{d+1}\times \mathbb{R})} \right)\|f\|_{L^2(\mathbb{R}^{d+1})}.
\end{align*}
 \end{proposition}
We can obtain Proposition \ref{prop03} by combining
 Lemma \ref{prop1} and the well-known almost orthogonality in Lemma \ref{lem442} below.  
 \begin{lemma}\label{prop1}
  If $|\ell|\ge C j$ with $C\ge 1$, then for $\tilde{\mathcal{A}}_{j,k}$ in (\ref{3990})  there  is  $c>0$ so that
\begin{align}
\left\| 2^{k/2}\tilde{\mathcal{A}}_{j,k}\mathcal{P}_{j+k+\ell} \right\|_{ L^2(\mathbb{R}^{d+1})\rightarrow L^2(\mathbb{R}^{d+1}\times \mathbb{R})}  &\lesssim 2^{-c|\ell|}   \label{00453}
\end{align}
where the implicit bound is independent of $j,k,\ell$. For all integers $k_1,k_2$,
 \begin{align}
\left\|  \mathcal{P}_{k_1} \mathcal{P}_{k_2}^*  \right\|_{ L^2(\mathbb{R}^{d+1})\rightarrow L^2(\mathbb{R}^{d+1})} &\lesssim   2^{-c|k_1-k_2|}. \label{045312}
\end{align}
\end{lemma}
\begin{proof}[Proof of Lemma \ref{prop1}]
We place its proof (standard)  in Appendix (Sec \ref{Sec114}).
\end{proof}
  \begin{lemma}\label{lem442}
Let $M_1,M_2$ be two measure spaces and let $\{T_k\}_{k\in\mathbb{Z}}$ be a family of the operators   $T_k:L^2(M_1)\rightarrow L^2(M_2)$.  If $\|T_{k_1}T_{k_2}^*\|_{op}\le C(|k_1-k_2|) $ with $\sum_{k\in \mathbb{Z}}  \sqrt{C(|k|)}$ being finite, then   for all $f\in L^2(M_1)$,
\begin{align}
\left\|  \left( \sum_{k\in \mathbb{Z}} |T_k f|^2\right)^{1/2}\right\|_{L^2(M_2)}\lesssim     \left(  \sup_k \|T_k\|_{op} \sum_{k\in \mathbb{Z}}   \sqrt{C(|k|)}\right)^{1/2}  \|f\|_{L^2(M_1)}.\label{8kg}
\end{align}
\end{lemma}

\subsection{Dilation and Localization}
Now that Proposition \ref{prop03} has been proven, the next step is to estimate the operator norm of $2^{k/2}\tilde{\mathcal{A}}_{j,k}$ uniformly in $k$. To accomplish this, it is sufficient to consider the case when $k=0$, as in Lemma \ref{prop23}.
 \begin{lemma}[Dilation]\label{prop23}
For every  $(j,k)\in\mathbb{Z}_+\times \mathbb{Z}$ and  $p\ge 1$,   
\begin{align}\label{yb1}
\left\|2^{k/p}\tilde{\mathcal{A}}_{j,k}\right\|_{L^p(\mathbb{R}^{{d+1}})\rightarrow L^p(\mathbb{R}^{{d+1}}\times\mathbb{R})} &= \left\|\tilde{\mathcal{A}}_{j,0}\right\|_{L^p(\mathbb{R}^{d+1})\rightarrow L^p(\mathbb{R}^{d+1}\times [1,2])},\\
\left\| 2^{k/p}\tilde{\mathcal{A}}_{j,k}\mathcal{P}_{j+k+\ell} \right\|_{ L^p(\mathbb{R}^{d+1})\rightarrow L^p(\mathbb{R}^{d+1}\times \mathbb{R})}&  = \left\|  \tilde{\mathcal{A}}_{j,0}\mathcal{P}_{j+\ell} \right\|_{ L^p(\mathbb{R}^{d+1})\rightarrow L^p(\mathbb{R}^{d+1}\times [1,2])}. \nonumber
\end{align}
 \end{lemma}
  \begin{proof}
It follows from the nonisotropic dilation  in the appendix (Sec  \ref{app11.2}).
 \end{proof}
 \subsection{Uniform $L^2$ Estimate}
 According to Lemma \ref{prop23}, we can fix $k=0$ by restricting $t$ to the interval $[1,2]$. This restriction also allows us to localize $|x|\leq 10$ in $\mathbb{R}^d$.
To obtain the $L^2$ estimate of $\tilde{\mathcal{A}}_{j,0}$, it is sufficient to fix $\xi_{d+1}=\lambda$ and obtain an uniform $L^2(\mathbb{R}^d)$ bound in $\lambda$ for the scaled operators $\mathcal{A}_{j}^{\lambda}$.
\begin{lemma}
Set   $\mathcal{A}_{j}^{\lambda}$ mapping $g\in L^2(\mathbb{R}^{d})\rightarrow \mathcal{A}_{j}^{\lambda}g\in L^2(\mathbb{R}^{d}\times \mathbb{R})$:
\begin{align*}
 \mathcal{A}_{j}^{\lambda}g (x,t )&:=\lambda^{d/2} \chi\left(t\right) \psi(x)   \int_{\mathbb{R}^d}
e^{2\pi i\lambda  \langle x,  \xi \rangle } \widehat{d\tilde{\sigma}_{j}}\left(t\lambda (\xi+ Ax)\right)   \widehat{g}(  \xi)d\xi. \end{align*} 
Then for $\tilde{\mathcal{A}}_{j,k}$ in (\ref{3990}) with $k=0$, we have
  \begin{align}\label{011}
\left\|\tilde{\mathcal{A}}_{j,0}\right\|_{L^2(\mathbb{R}^{d+1})\rightarrow L^2( \mathbb{R}^{{d+1}}\times\mathbb{R})}\le  \sup_{\lambda} \left\| \mathcal{A}_{j}^{\lambda} \right\|_{L^2(\mathbb{R}^d)\rightarrow L^2(\mathbb{R}^d\times \mathbb{R})}.
  \end{align} 
\end{lemma}

\begin{proof}[Proof of (\ref{011})]
Let $ \widehat{g_{\lambda}}(\xi)=\lambda^{d/2}\widehat{f}(\lambda \xi,\lambda) $ and change  $\xi\rightarrow \lambda \xi$ to write (\ref{3990}) whose symbol does not depend on $x_{d+1}$, as
$$ \tilde{\mathcal{A}}_{j,0}f(x,x_{{d+1}},t)= \int_{\mathbb{R}}
 \left[ \mathcal{A}_{j}^{\lambda}g_{\lambda} (x,t ) \right]   e^{2\pi i x_{d+1} \lambda}d\lambda.  $$
 Then 
one can  apply  the Plancherel theorem  for $x_{{d+1}}$  whose frequency    $\lambda$,   
\begin{align*}
\int\left[\int_{\mathbb{R}} |\tilde{\mathcal{A}}_{j,0}f(x,x_{{d+1}},t)|^2dx_{d+1} \right]dxdt  &= \int\left[\int_{\mathbb{R}} |\mathcal{A}_{j}^{\lambda}g_{\lambda} (x,t )|^2 d\lambda \right]dxdt\\
&\le \int_{\mathbb{R}} \| \mathcal{A}_{j}^{\lambda}\|_{L^2(\mathbb{R}^d)\rightarrow L^2(\mathbb{R}^{d}\times \mathbb{R})}^2  \|\widehat{g_{\lambda}}(\cdot)\|_{L^2(d\xi)}^2 d\lambda 
\end{align*}
where $\int \|\widehat{g_{\lambda}}(\cdot)\|_{L^2(d\xi)}^2d\lambda=\|f\|_{L^2(\mathbb{R}^{d+1})}^2$. This yields (\ref{011}).
\end{proof}

 \subsection{Wave Propagating Operator}
For the $L^2$ estimate of  $\mathcal{A}_{j}^{\lambda} $, from (\ref{i7}), it suffices to   treat the main part  
$$  
\widehat{d\sigma_{j}}\left(t\lambda (\xi+ Ax)\right) \psi(x)\chi(t)\ \text{where  $\widehat{d\sigma_j}(\xi)= \widehat{d\sigma}(\xi)\chi\left(\frac{\xi}{2^j}\right)$    in (\ref{kc1})},
$$   
of the symbol    $\widehat{d\tilde{\sigma}_{j}}\left(t\lambda (\xi+ Ax)\right) \psi(x)\chi(t)$ of $ \mathcal{A}_{j}^{\lambda}  $. Using the asymptotics of the Bessel function $ J_{(d-2)/2}$ in (\ref{p54})
on   page 356 in \cite{St}, 
 write
  $\widehat{d\sigma_j}(\xi)= \widehat{d\sigma}(\xi)\chi\left(\frac{\xi}{2^j}\right)$    as  
\begin{align}\label{kc3}
 \widehat{d\sigma_j}(\xi)=\sum_{\ell =0}^N \left( a_\ell e^{  2\pi i|\xi|}+   b_\ell e^{  -2\pi i|\xi|} \right)  \frac{2\pi\chi_\ell\left(\frac{\xi}{2^j}\right) }{2^{j\left(\frac{d-1}{2}+\ell\right) }}+O\left( \frac{\tilde{\chi}\left(\frac{\xi}{2^j}\right)}{2^{j\left(\frac{d-1}{2}+N+\frac{1}{2}\right) }   }\right) .
\end{align}
Here   $\chi_\ell (\xi):= \chi(\xi) |\xi|^{-\left(\frac{d-1}{2}+\ell\right)}$   and  $a_0=(\frac{\pi}{2})^{-1/2}e^{-i(\frac{n-1}{4}\pi)}$ and $b_0=(\frac{\pi}{2})^{-1/2}e^{i(\frac{n-1}{4}\pi)}$.
The error term of (\ref{kc3}) for $  \| \mathcal{A}_{j}^{\lambda} \|_{L^2(\mathbb{R}^d)\rightarrow L^2(\mathbb{R}^d\times \mathbb{R})}$, gives a better bound  $O(2^{-Nj})$, since
\begin{align*}
& \int \lambda^{d/2} \chi\left(t\right) \psi(x)\frac{\tilde{\chi}\left(\frac{\left(t\lambda (\xi+ Ax)\right)}{2^j}\right)}{2^{j\left(\frac{d-1}{2}+N+\frac{1}{2}\right) }   } dxdt=O\left(\frac{ \lambda^{d/2}}{2^{j\left(\frac{d-1}{2}+N+\frac{1}{2}\right) }   } \right)  \ \text{and}\\
 &\int \lambda^{d/2} \chi\left(t\right) \psi(x) \frac{\tilde{\chi}\left(\frac{\left(t\lambda (\xi+ Ax)\right)}{2^j}\right)}{2^{j\left(\frac{d-1}{2}+N+\frac{1}{2}\right) }   } d\xi =O\left(\frac{ \lambda^{d/2}  (2^j/\lambda)^d}{2^{j\left(\frac{d-1}{2}+N+\frac{1}{2}\right) }   } \right). 
 \end{align*}
Therefore, to estimate  $  \| \mathcal{A}_{j}^{\lambda} \|_{L^2(\mathbb{R}^d)\rightarrow L^2(\mathbb{R}^d\times \mathbb{R})}$, we   focus on only the first term in (\ref{kc3}) of  $    \frac{e^{  2\pi i|\xi|} }{2^{j(d-1)/2}} \chi \left(\frac{\xi}{2^j}\right) $ 
and set $ \mathcal{T}_{j}^{\lambda} $ satisfying $\mathcal{A}_{j}^{\lambda}\approx \frac{\mathcal{T}_{j}^{\lambda}}{2^{(d-1)j/2}}$, defined by
\begin{align}
\label{5rg}\mathcal{T}_{j}^{\lambda}g (x,t )&= \lambda^{d/2}  \int_{\mathbb{R}^d}
e^{2\pi i\lambda \left(\langle x,  \xi \rangle+t| \xi+Ax|\right)} \chi\left(t\right) \chi\left(\frac{\lambda t|\xi+Ax|}{2^j}\right) \widehat{g}(  \xi)d\xi\end{align}   
which is interpreted as an analog of a wave propagating operator in $\mathbb{G}_A$.

 \section{Estimates of Oscillatory Integral Operators}

  \subsection{Estimate for Oscillatory Integral Operator $\Rightarrow$ Maximal Estimate}
\begin{proposition}\label{prop44}
Consider $\mathcal{T}_{j}^{\lambda} $ in (\ref{5rg}).  Suppose  that   there exists $C_\epsilon>0$ independent of $j$  for sufficiently small $\epsilon>0$ such that
\begin{align}\label{kc4}
  \sup_{\lambda} \left\| \mathcal{T}_{j}^{\lambda} \right\|_{L^2(\mathbb{R}^d)\rightarrow L^2(\mathbb{R}^d\times [1,2])}\le C_{\epsilon}   2^{\epsilon j} 2^{s_Aj/2}\ \text{ for every $j\ge 0$}.
\end{align}
Then   there are two positive constants $C_p$ and $C$ such that
\begin{align}
\|\mathcal{M}_{S^{d-1}(A)}\|_{L^p(\mathbb{R}^{d+1})\rightarrow L^p(\mathbb{R}^{d+1})}&\le C_p\ \text{if $p>\frac{d-s_A}{d-s_A-1}$ and $d-s_A>2$.}\label{2.24}\\
\|\mathcal{M}^{\delta}_{S^{d-1}(A)}\|_{L^2(\mathbb{R}^{d+1})\rightarrow L^2(\mathbb{R}^{d+1})}&\le C \delta ^{(d-2)/2}  (\delta^{-1})^{ s_A/2+\epsilon} \label{2.25}.
\end{align}
Note    $s_A=m_A-1$ in Main Theorems \ref{main1},\ref{main2} and (\ref{004}).  
\end{proposition}
\begin{proof}[Proof of  Proposition \ref{prop44}]
From    (\ref{115a}),(\ref{0453})   and Proposition  \ref{prop03}  together with  (\ref{yb1}),(\ref{011}) and (\ref{kc4}),  we obtain that  
\begin{align}\label{2.26}
 \left\|\mathcal{M}_{j} \right\|_{L^2(\mathbb{R}^{d+1})\rightarrow L^2(\mathbb{R}^{d+1})} & \lesssim 2^{j/2}  
\sup_k\left\|2^{k/2}\tilde{\mathcal{A}}_{j,k}\right\|_{L^2(\mathbb{R}^{{d+1}})\rightarrow L^2(\mathbb{R}^{{d+1}}\times\mathbb{R})} \nonumber\\
& \lesssim  2^{j/2} \left\| \tilde{\mathcal{A}}_{j,0}\right\|_{L^2(\mathbb{R}^{{d+1}})\rightarrow L^2(\mathbb{R}^{{d+1}}\times\mathbb{R})}  \lesssim 2^{j/2}\sup_{\lambda}  \left\|\mathcal{A}_j^{\lambda}\right\|_{L^2(\mathbb{R}^d)\rightarrow L^2(\mathbb{R}^{d}\times [1,2])} \nonumber\\
&\le2^{j(1/2-(d-1)/2)}  \sup_{\lambda} \left\| \mathcal{T}_{j}^{\lambda} \right\|_{L^2(\mathbb{R}^d)\rightarrow L^2(\mathbb{R}^d\times [1,2])}   \\ 
&\le 2^{j(1/2-(d-1)/2)} 2^{\epsilon j} 2^{s_Aj/2}. \nonumber
\end{align}
By using the interpolation $p_1=2$ in (\ref{2.26}) and $p_0=1+\eta$ with $0<\eta\ll 1$ in (\ref{kcc2}),
\begin{align*}
  \|\mathcal{M}_{S^{d-1}(A)}\|_{L^p\rightarrow L^p}&\le \sum_{j=0}^\infty \|\mathcal{M}_{j}\|_{L^p\rightarrow L^p}  \\
  & \lesssim\sum_{j\ge 0} 2^{\epsilon  j}  2^{j(1-\theta)}2^{j(s_A/2-(d-2)/2)\theta} \\
& \lesssim\sum_{j\ge 0} 2^{(\epsilon+(d+2)\eta) j} 2^{ j[1-(1-1/p)(d-s_A)]} 
\end{align*}
for $\frac{1}{p}=\frac{(1-\theta)}{p_0}+\frac{\theta}{p_1}$, namely, $\theta=2(1-1/p)-O(\eta)$ where $0\le O(\eta)\le 2\eta.  $
Hence, if $p>\frac{(d- s_A)}{(d- s_A)-1}$ with  $d-s_A>2$, that is, $(1-1/p)(d-s_A)>1$, then one can  take $\epsilon+(d+2)\eta$ smaller than $(1-1/p)(d-s_A)-1 $ to sum over $j$ to obtain the boundness of the spherical maximal operator $\mathcal{M}_{S^{d-1}(A)}$  in $L^p(\mathbb{R}^{d+1})$ in (\ref{2.24}).
 Next, we claim that  (\ref{kc4})    implies   (\ref{2.25}). Write the  measure  supported on the annulus $S^{d-1}_\delta$   with  
 $\delta=2^{-j}$ as
$
 \left(d\sigma*2^{dj}\psi (2^j\cdot )\right)\left(y\right) $ for $y\in \mathbb{R}^d$  where $*$ is the Euclidean convolution in $\mathbb{R}^d$. 
Take its Fourier transform  in the Euclidean space $\mathbb{R}^2$ and make a dyadic decomposition to
 $ \widehat{d\sigma}(\xi) \widehat{\psi}\left(  \frac{\xi}{2^j}\right)  $. Then we have the similar bound of  $\left\|\mathcal{M}_{j} \right\|_{L^2(\mathbb{R}^{d+1})\rightarrow L^2(\mathbb{R}^{d+1})} $  in (\ref{2.26}) with a smaller tail term  so that 
\begin{align*}
\|\mathcal{M}^{\delta}_{S^{d-1}(A)}\|_{L^2\rightarrow L^2} &\lesssim \sum_{\ell: 2^0\le 2^{\ell}\le 2^j} \| \mathcal{M}_\ell\|_{L^2\rightarrow L^2}\approx 2^{j(1/2-(d-1)/2)} 2^{\epsilon j} 2^{s_Aj/2} 
\end{align*}   
which is $(\delta^{-1})^{(1/2-(d-1)/2)+s_A/2+\epsilon}$ as  $\delta=2^{-j}$. So, we proved (\ref{2.25}).
\end{proof}

\subsection{Reduction to the case $\lambda\approx 2^j$ in (\ref{kc4})}\label{sec5}  
Let $A\in G_{d}(\mathbb{R})$.  
The change of variable $x\rightarrow A^{-1}(x)$ for   $\mathcal{T}_{j}^{\lambda}g$   in  (\ref{5rg}), yields\begin{align*}
\mathcal{T}_{j}^{\lambda}g(A^{-1}(x),t )& =\frac{ \lambda^{d/2}   }{\det(A)} \int_{\mathbb{R}^d}
e^{2\pi i\lambda \left(\langle A^{-1}(x), \xi \rangle+t| \xi+x|\right)}
 \chi\left(t\right)  \psi(x)  \chi\left(\frac{\lambda t|\xi+x|}{2^j}\right) \widehat{g
 }(  \xi)d\xi.
  \end{align*}
 We are allowed to replace $A^{-1}$ with $A$ in the above   and rewrite it as
\begin{align}\label{5ppa}
  \mathcal{T}_{j}^{\lambda}g(x,t )  = \lambda^{d/2}    \int_{\mathbb{R}^d}
e^{2\pi i\lambda \left(\langle Ax, \xi \rangle+t| \xi+x|\right)}
 \chi\left(t\right) \psi(x)  \chi\left(\frac{\lambda t|\xi+x|}{2^j}\right) \widehat{g
 }(  \xi)d\xi.
\end{align}

 \begin{proposition}[Extreme $\lambda$] \label{pr02}
  Let $A\in G_{d}(\mathbb{R})$ and  denote $ \mathcal{T}_{j}^{\lambda,t}g(x)  = \mathcal{T}_{j}^{\lambda}g(x,t ) $ for each fixed $t$ in (\ref{5ppa}). Then there is $C>0$ such that
\begin{align}
&\left\|\mathcal{T}_j^\lambda   \right\|_{L^2(\mathbb{R}^d)\rightarrow L^2(\mathbb{R}^{d}\times [1,2]) }\le\sup_{t} \|  \mathcal{T}_{j}^{\lambda,t}  \|_{L^2(\mathbb{R}^d)\rightarrow L^2(\mathbb{R}^d)}\le  C    \ \text{ if $\left|\frac{2^j}{\lambda}\right|\gg 1$, }
\label{546}\\
& \left\|\mathcal{T}_j^\lambda  \right\|_{L^2(\mathbb{R}^d) \rightarrow L^2(\mathbb{R}^{d}\times [1,2]) }\lesssim \left|\frac{2^j}{\lambda}\right|^{d/2} \le 1\  \  \text{ if $\left|\frac{2^j}{\lambda}\right|\ll  1$.}\label{545}
\end{align}
\end{proposition}

\begin{proof}[Proof of (\ref{546})]
Using    $(x,\xi)\rightarrow \left(\left(\frac{2^j}{\lambda}\right)x,\left(\frac{2^j}{\lambda}\right) \xi\right)$ in (\ref{5ppa}), it suffices to treat:
\begin{align*}  
 \mathcal{T}_{j}^{\lambda}g( x,t )  =\left(\frac{2^j}{\lambda}\right)^{d}  \lambda^{d/2}  \int_{\mathbb{R}^d}
e^{2\pi i\lambda  \left[\left(\frac{2^j}{\lambda}\right)^2  \langle Ax, \xi \rangle+   \left(\frac{2^j}{\lambda}\right)t| \xi+x|  \right]  }
 \chi\left(t\right) \chi\left(  t|\xi+x|\right) \widehat{g
 }(  \xi)d\xi.
\end{align*}
Rewrite the above oscillatory part as
$$    \frac{2^{2j}}{\lambda}   \phi(x,t,\xi)\ \text{where}\ \phi(x,t,\xi)= \langle Ax, \xi \rangle+   \frac{\lambda}{2^{j}} t| \xi+x|. $$ 
 Fix $|t|\approx 1$. Then 
  the determinant of the $d\times d$ mixed hessian of   $\phi(x,t,\xi)$ is
\begin{align*}
\det( [\phi_{x_i\xi_j}''(x,t,\xi)])=\det(A)+O\left( \frac{\lambda}{2^j} \right)\approx \det(A)\gtrsim 1
\end{align*}
which follows from the multilinearity of $\det$, with   $ \frac{\lambda}{2^j} \ll 1$. This enables us to apply the   H\"{o}rmander theorem   to obtain that
 $$\|  (\mathcal{T}_{j}^{\lambda}g)(\cdot,t) \|_{L^2(\mathbb{R}^d)}\le C\left(\frac{2^j}{\lambda}\right)^{d}  \lambda^{d/2} \left( \frac{2^{2j}}{\lambda} \right)^{-d/2}  \|g\|_{L^2(\mathbb{R}^d)}=C\|g\|_{L^2(\mathbb{R}^d)} $$
with $C$  independent of $t$. Integrate LHS with respect to $dt$ to obtain (\ref{546}).
\end{proof} 
We shall utilize the H\"{o}rmander Theorem   in the following  form. 
 \begin{lemma}[H\"{o}rmander]\label{th51h} 
Let $\phi\in C^\infty$ and $\psi$ be compactly supported and let $$\mathcal{S}^\lambda f(x,t):=\int_{\mathbb{R}^d} e^{i\lambda \phi(x,t,\xi)} \psi(x,t,\xi)f(\xi)d\xi\ \text{where $(x,t)\in \mathbb{R}^d\times \mathbb{R}$}. $$
Denote by $B_i(x,t,\xi)$ be  $d\times d$ sub-matrices of the $(d+1)\times d$ mixed hessian   $[\phi]''_{((x,t),\xi)((x,t), \xi)}$.  Suppose that $\sum_{i=1}^{d+1}\det( B_i(x,t,\xi)|\ge c>0$.  Then 
\begin{align}\label{2010}
\|  \mathcal{S}^\lambda\|_{L^2(\mathbb{R}^d)\rightarrow L^2(\mathbb{R}^{d+1})}\le C\lambda^{-d/2}. 
\end{align}
 \end{lemma}
 \begin{proof}[Proof of (\ref{545})]
This is the case $    \frac{2^j}{\lambda}\ll 1$. For the same dilation,  we  work with
\begin{align}\label{445}
 \mathcal{T}_{j}^{\lambda}g( x,t )  =\left(\frac{2^j}{\lambda}\right)^{d}  \lambda^{d/2}  \chi\left(t\right)\psi(x)    \int_{\mathbb{R}^d}
e^{2\pi i 2^j  \phi(x,t,\xi)   }
 \chi\left(  t|\xi+x|\right) \widehat{g
 }(  \xi)d\xi
\end{align}
where 
$$ \phi(x,t,\xi)=\frac{2^j}{\lambda}  \langle Ax, \xi \rangle+  t| \xi+x|. $$
Let $u=(x+\xi)$ and $u_i=x_i+\xi_i$. Then   $(d+1)\times d$   mixed hessian matrix is
\begin{align*} 
  \left(\begin{matrix}
\phi_{x_1\xi_1}&\cdots&\phi_{x_1\xi_d} \\
\phi_{x_2\xi_1}&\cdots&\phi_{x_2\xi_d}\\
\vdots\\
\phi_{x_d\xi_1}&\cdots&\phi_{x_d\xi_d}\\
\phi_{t\xi_1}&\cdots&\phi_{t\xi_d}\\
 \end{matrix}  \right)=  \left(\left(\frac{2^j}{\lambda} A+\frac{t}{|u|}\right) I\begin{matrix}
 -\frac{tu_1}{|u|^2}\frac{u}{|u|}\\
 -\frac{tu_2}{|u|^2}\frac{u}{|u|}\\
 \qquad \vdots\\
 -\frac{tu_d}{|u|^2}\frac{u}{|u|}\\
\frac{u}{|u|}\qquad\qquad \\
 \end{matrix}  \right)\sim  \left(\begin{matrix}
\frac{2^j}{\lambda} A+\frac{t}{|u|}I\\
\frac{u}{|u|}\\
 \end{matrix}  \right).
\end{align*}
where $\det(\frac{2^j}{\lambda} A+\frac{t}{|u|}I)\approx 1$ due to $|\frac{2^j}{\lambda}|\ll 1$.  
This enables us to apply the  H\"{o}rmander theorem in (\ref{2010}) to obtain that
 $$\|  \mathcal{T}_{j}^{\lambda}  \|_{L^2(\mathbb{R}^d)\rightarrow L^2(\mathbb{R}^{d+1})}\le C\left(\frac{2^j}{\lambda}\right)^{d}  \lambda^{d/2}   2^{-jd/2}  =C\left(\frac{2^j}{\lambda}\right)^{d/2} . $$
 This yields (\ref{545}).
\end{proof}
We now have proved (\ref{546}) and (\ref{545}). To prove  Main Theorems, there remains   
     the case  $\lambda\approx 2^j$ of (\ref{kc4}) where $s_A=m_A-1$. So, we set $\mathcal{T}^{\lambda}_j $ with $\lambda\approx 2^j$ as
\begin{align*} 
 \mathcal{T}_j g(x,t):= 2^{dj/2}  \int_{\mathbb{R}^d}
e^{2\pi i 2^j\phi(x,t,\xi)}\psi(x,t,\xi)
 \widehat{g}(  \xi)d\xi
\end{align*}
where  
 $\phi(x,t,\xi):=\langle x, A(\xi)\rangle+t|\xi+x|$ and $\psi(x,t,\xi)=\chi\left(t\right)\psi\left(x\right)\chi\left(  t|\xi+x| \right)$. Then
it suffices to claim that under the hypotheses of Main Theorems,
\begin{align}\label{apl}
&\left\| \mathcal{T}_j\right\|_{L^2(\mathbb{R}^d)\rightarrow L^2(\mathbb{R}^d\times  [1,2])}\lesssim_{\epsilon}   2^{j(s_A/2+\epsilon)}\ \text{where $s_A=m_A-1$}
\end{align}
where we take $m_A:=1$ if $\Lambda(A)\cap \mathbb{R}=\emptyset$.

 \subsection{Proof of Main Theorem 1} \label{secc6}

  \begin{proposition}\label{th44}
 Suppose that $\Lambda(A)\cap \mathbb{R}=\emptyset$. Then  it holds that 
\begin{align} \label{apl5}
 \left\| \mathcal{T}_j  \right\|_{ L^2(\mathbb{R}^{d})\rightarrow L^2(\mathbb{R}^{d}\times \mathbb{R})} &\lesssim 1.
 \end{align}
Thus, we obtain the results  in (\ref{apl}) for $m_A=1$, which proves Main Theorem 1.
\end{proposition}
\begin{proof}[Proof of (\ref{apl5})]
Let $u=(x+\xi)$ and $u_i=x_i+\xi_i$. As we did above,   compute the   $(d+1)\times d$ mixed hessian matrix $[ \phi_{(x_1\cdots x_d t)(\xi_1\cdots \xi_d)}''(x,t,\xi)]$ of $\phi$  as
\begin{align}\label{300}
   \left(\begin{matrix}
\phi_{x_1\xi_1}&\cdots&\phi_{x_1\xi_d} \\
\phi_{x_2\xi_1}&\cdots&\phi_{x_2\xi_d}\\
\vdots\\
\phi_{x_d\xi_1}&\cdots&\phi_{x_d\xi_d}\\
\phi_{t\xi_1}&\cdots&\phi_{t\xi_d}\\
 \end{matrix}  \right)=  \left(\left(A+\frac{t}{|u|}\right) I\begin{matrix}
 -\frac{tu_1}{|u|^2}\frac{u}{|u|}\\
 -\frac{tu_2}{|u|^2}\frac{u}{|u|}\\
 \qquad \vdots\\
 -\frac{tu_d}{|u|^2}\frac{u}{|u|}\\
\frac{u}{|u|}\qquad\qquad \\
 \end{matrix}  \right)\sim  \left(\begin{matrix}
A+\frac{t}{|u|}I\\
\frac{u}{|u|}\\
 \end{matrix}  \right).
\end{align} 
Since $A$ has no real eigenvalue, we have for all $t,u $ with $|t|\approx |u|\approx 1$,
 $$|\det (A+\frac{t}{|u|}I)|\gtrsim 1\ \text{in (\ref{300})}.$$
This with (\ref{2010})  yields the result.
\end{proof}

\section{Proof of Main Theorem \ref{main2}} 
 Let $d\ge 2$ and $n(\mu)=1$, that is,  $m_a(\mu)=m_g(\mu)$ for every $\mu\in \Lambda(A)\cap \mathbb{R}$. Recall $m_A$ the  maximal   multiplicity  in (\ref{m11}).  Then we shall claim that
\begin{align} \label{51}
 \left\| \mathcal{T}_j  \right\|_{ L^2(\mathbb{R}^{d})\rightarrow L^2(\mathbb{R}^{d}\times \mathbb{R})} &\lesssim_{\epsilon}  2^{(m_A-1)j/2}.
 \end{align}
This yields (\ref{apl}) for $s_A=m_A-1$, giving a desired   bound  of  Main Theorem \ref{main2}.
\subsection{Away from Eigenvalues}
  Take  $\rho$ satisfying
 \begin{align}\label{cx21}
   0<  \rho&\le 2^{-10d}\min\left\{1, |\mu_{i}-\mu_{j}|: \mu_{i},\mu_{j}\in \Lambda(A)\cap\mathbb{R}  \right\}.     
\end{align}  
Then  the supports of the functions  $s\rightarrow\psi\left(\frac{\mu +s}{\rho}\right) $ for $\mu \in  \Lambda(A)\cap \mathbb{R}$ are disjoint.   Split $ \mathcal{T}_j =\mathcal{T}_{j,\rm{good}}+\sum_{\mu \in\Lambda(A)\cap \mathbb{R}} \mathcal{T}_{j,\mu}$ where $ \mathcal{T}_{j,\rm{good}} g(x,t)$ is defined as
\begin{align*} 
& 2^{dj/2}   \int 
e^{ i 2^j\phi(x,t,\xi)} \psi\left(x,t,\xi\right)\left[1-\sum_{\mu \in\Lambda(A)\cap \mathbb{R}}\psi\left(\left( \mu  +  \frac{t}{|\xi+x|}\right)\frac{1}{\rho}\right) \right] \widehat{g}(  \xi)d\xi\ \text{and}\\
& \mathcal{T}_{j,\mu} g(x,t)= 2^{dj/2}   \int 
e^{ i 2^j\phi(x,t,\xi)} \psi\left(x,t,\xi\right)
 \psi\left(\left( \mu  +  \frac{t}{|\xi+x|}\right)\frac{1}{\rho}\right)\widehat{g}(  \xi)d\xi.
\end{align*}
From the support condition of $\mathcal{T}_{j,\rm{good}}$, one can observe that  the mixed hessian of $\phi$ is of rank $d$ in (\ref{300}). Thus one can apply (\ref{2010}) to obtain that
\begin{align*} 
 \left\| \mathcal{T}_{j,\rm{good}}  \right\|_{ L^2(\mathbb{R}^{d})\rightarrow L^2(\mathbb{R}^{d}\times \mathbb{R})} &\lesssim 1.
 \end{align*}
Hence to prove (\ref{51}), there remains to prove that
  for   $\mu\in\Lambda(A)\cap\mathbb{R}$,
\begin{align}\label{cxx1}
 \left\| \mathcal{T}_{j,\mu}  \right\|_{ L^2(\mathbb{R}^{d})\rightarrow L^2(\mathbb{R}^{d}\times \mathbb{R})} &\lesssim  2^{j\epsilon}2^{(m_g(\mu)-1)j/2}\ \text{if $n(\mu)=1$}.
\end{align}

 \subsection{Kernel of  $T^*T$}
To show (\ref{cxx1}),
 write the  integral kernel   of $[\mathcal{T}_{j,\mu}]^*\mathcal{T}_{j,\mu}$    as
\begin{align*} 
K(\xi,\eta) &=  2^{jd} \int_{\mathbb{R}^{d+1}} e^{i 2^j \Phi(x,t,\xi,\eta) }\Psi_\mu(x,t,\xi,\eta) dxdt
 \end{align*}
with    $\Phi(x,t,\xi,\eta)=\langle A(\xi-\eta), x\rangle +t\left(|\xi+x|-|\eta+x|\right)$ and its amplitude $\Psi_\mu$ is
   $$ \chi(t)\psi\left(x\right)  \chi\left( |\xi+x| \right) \chi\left( |\eta+x| \right)  \psi\left(\left( \mu  +  \frac{t}{|\xi+x|}\right)\frac{1}{\rho}\right) 
   \psi\left(\left( \mu  +  \frac{t}{|\eta+x|}\right)\frac{1}{\rho}\right).$$
Note that   $|\nabla_{x,t}\Phi_\mu|\le C$ with $C$ depending only $A$, which is harmless. 
 Split $K= K_{1}+  K_{2}+K^{\mu}_{\rm{main}} $   where the small error terms are
\begin{align*}
& K_{1}(\xi,\eta)= 2^{jd}\int_{\mathbb{R}^{d+1} } e^{i 2^j
\Phi(x,t,\xi,\eta)}
 \Psi_\mu(x,t,\xi,\eta)\psi^c\left(\frac{ |\xi+x| -|\eta+x|
  }{2^{(\epsilon-1)j} }\right)  dxdt    \\
&   K_{2}(\xi,\eta)= 2^{jd}\int_{\mathbb{R}^{d+1} } e^{ i 2^j
\Phi(x,t,\xi,\eta)}
 \Psi_\mu(x,t,\xi,\eta)\psi\left(\frac{ |\xi+x| -|\eta+x|
  }{2^{(\epsilon-1)j} }\right)\psi\left( \frac{|\xi-\eta|}{2^{(3\epsilon-1)j}} \right)   dxdt
\end{align*}
with $\psi^c=1-\psi$, and the main term $K^{\mu}_{\rm{main}}(\xi,\eta)$ is given by
\begin{align}
 2^{jd}\int  e^{i 2^j
\Phi(x,t,\xi,\eta)} \Psi_\mu(x,t,\xi,\eta) \psi\left(\frac{ |\xi+x| -|\eta+x|
  }{2^{(\epsilon-1)j} }\right)  \psi^c\left( \frac{|\xi-\eta|}{2^{(3\epsilon-1)j}}   \right)dxdt.  \label{pc98}
 \end{align}
Using  $|\partial_t\Phi|=| |\xi+x|-|\eta+x||\gtrsim 2^{(\epsilon-1)j} $ for the integration by parts  along $dt$,
 \begin{align*}
  \sup_{\eta}\int_{\mathbb{R}^{d}} \left|K_{1}(\xi,\eta) \right|d\xi     &\lesssim  2^{-Nj}.   
  \end{align*}
By measuring the small frequency region $d\xi   = O(2^{(3\epsilon-1)d j} )$,
\begin{align*}
\sup_{\eta}\int \left|K_{2}(\xi,\eta)\right|d\xi&\lesssim  2^{j3d\epsilon}. 
\end{align*}
There remains to show  under the condition
 $n(\mu)=1$ for  $\mu\in \Lambda(A)\cap \mathbb{R}$:
\begin{align}
\sup_{\eta}\int |K^\mu_{\rm{main}}(\xi,\eta)|d\xi  \lesssim   2^{(m_g(\mu)-1)j}   2^{c\epsilon j} \ \text{for $c>0$ independent of $j$.} \label{cx1}
\end{align}
Here one can switch the roles of $\xi$ and $\eta$ when applying Schur's test.
\subsection{Two Preliminary Lemmas}
The following result is a key lemma, although it is elementary in linear algebra.

\begin{lemma}\label{lem3.3} 
Let $A\in G_{d}(\mathbb{R})$ and   $\mu\in \Lambda(A)\cap  \mathbb{R}$.  Then   for every $\xi\in\mathbb{R}^d$,
\begin{align}\label{pv61}
 &  |(A-\mu I)\xi|\approx|\mathcal{P}_{\mathcal{E}^{\perp}(\mu)}(\xi) |
\end{align}
where  $\mathcal{P}_S$  is a projection   to a subspace $S\subset \mathbb{R}^d$.  If  $n(\mu)=1$, then it holds that
\begin{align}\label{5252}
  \mathbb{R}^d=\mathcal{C}(A-\mu I) \oplus \mathcal{E}(\mu),
\end{align}
so that  $(A-\mu I)\xi$ in (\ref{pv61}) is contained in $\mathcal{E}^{\perp}(\mu)$ and
\begin{align}
&  (A-\mu I)\xi=  \langle (A-\mu I)\xi, v\rangle v\ \text{where}\ v=\frac{\mathcal{P}_{\mathcal{E}^{\perp}(\mu)}((A-\mu I)\xi)}{|\mathcal{P}_{\mathcal{E}^{\perp}(\mu)}((A-\mu I)\xi)|}. \label{59}
\end{align}
  \end{lemma}
\begin{proof}[Proof of (\ref{pv61})]
It holds since $A-\mu I: \mathcal{E}^{\perp}(\mu)\rightarrow \mathcal{C}(A-\mu I)$ is an isomorphism.  
 \end{proof}

 \begin{proof}[Proof of (\ref{5252}) and (\ref{59})]
The assumption $n(\mu)=1$ implies    that   there is no  $v$ satisfying (i) $(A-\mu I)^2 v=0$   and (ii) $(A-\mu I) v\ne 0$ together. This yields $  \mathcal{E}(\mu) \cap \mathcal{C} (A-\mu I)=\{{\bf 0}\}$ and $\mathbb{R}^d=\mathcal{C}(A-\mu I) \oplus \mathcal{E}(\mu) $ by rank-nullity theorem, showing (\ref{5252}). Thus $(A-\mu I)\xi$ in (\ref{pv61}) is contained in $\mathcal{E}^{\perp}(\mu)$. So one can express $(A-\mu I)\xi=\langle (A-\mu I)\xi, v\rangle v$ for the unit vector $v$ in $\mathcal{E}^{\perp}(\mu)$ in (\ref{59}).
\end{proof} 

The next lemma is also elementary due to the Fubini Theorem. 
 \begin{lemma}
 Let  $E$ be a subspace of $\mathbb{R}^d$. Denote the projections of $\xi\in\mathbb{R}^d$ to $E$ and $E^{\perp}$ by $\xi_a=  \mathcal{P}_E(\xi)$ and $\xi_b=\mathcal{P}_{E^{\perp}}(\xi)$ respectively. Then  for $F\in L^1(\mathbb{R}^d)$, 
\begin{align}\label{215a}
\int_{\mathbb{R}^d} F(\xi)d\xi=\int_{\xi_{b }\in E^{\perp}} \left[\int_{\xi_a \in E} F(\xi_a+\xi_{b})d\xi_a\right]d\xi_{b}
\end{align}
where $d\xi_a$ and $d\xi_{b}$ are the  measures supported on  $E$  and $E^{\perp}$ respectively.
\end{lemma}
\begin{proof}[Proof of (\ref{215a})]
Let $\{v_1,\cdots,v_k,v_{k+1},\cdots,v_d\}$ be an orthonormal basis of $\mathbb{R}^d$ such that   $E=\text{span}(v_1,\cdots,v_k )$ and $E^{\perp}=\text{span}^{\perp}(v_{k+1},\cdots,v_d).$
Let $Q=(v_1|\cdots|v_d)$. Take $\xi=Q\zeta$. Then by Fubini,
\begin{align*}
\int F(\xi)d\xi  
&=\int_{  (\zeta_{k+1},\cdots,\zeta_d)\in\mathbb{R}^{d-k}}\left[\int_{  (\zeta_1,\cdots,\zeta_k)\in\mathbb{R}^k} F(Q\zeta)d\zeta_1\cdots d\zeta_k\right] d\zeta_{k+1}\cdots d\zeta_d 
\end{align*}
for  $Q\zeta=(\zeta_1v_1+\cdots+\zeta_kv_k)+(\zeta_{k+1}v_{k+1}+\cdots+\zeta_dv_d)=  \xi_a+\xi_b$, leading (\ref{215a}).
\end{proof}

\subsection{Proof of (\ref{cx1})} For $B(\eta):= \{\xi:  |\mathcal{P}_{\mathcal{E}^{\perp}(\mu)}(\xi-\eta)|\le   \frac{1}{\rho} 2^{(2\epsilon-1)j} \}$,  we first claim 
\begin{align}\label{081}
\int_{B(\eta)}\left| K^{\mu}_{\rm{main}}(\xi,\eta) \right|d\xi\lesssim 2^{c\epsilon j}2^{j(m_g(\mu)-1)})\ \text{for $c>0$ depending only on $d$} .
\end{align}
\begin{proof}[Proof of (\ref{081})]
From the support condition $|\xi-\eta|\gg ||\xi+x|-|\eta+x|| $ in (\ref{pc98}),
\begin{align*} 
&\nabla_x(|\xi+x|-|\eta+x|)= \frac{ (\xi-\eta) }{|\xi+x|} +(\eta+x)  \left( \frac{1}{|\xi+x|} - \frac{1}{|\eta+x|} \right)= |\xi-\eta| +O(2^{(\epsilon -1)j}).
\end{align*}
This enables us to apply  the sublevel-set estimates with respect to  $dx$ to have
\begin{align}\label{813}
\left| K^{\mu}_{\rm{main}}(\xi,\eta) \right| &\le 2^{jd}\int \left|\psi\left(\frac{ |\xi+x| -|\eta+x|
  }{2^{(\epsilon-1)j} }\right) \psi^c\left( \frac{|\xi-\eta|}{2^{(3\epsilon-1)j}}   \right)\right|dx\psi(\xi)\psi(\eta)  \\
& \lesssim\frac{ 2^{jd}\psi(\xi)\psi(\eta) }{2^{(1-\epsilon)j} |\xi-\eta| +1 } \nonumber
\end{align}
where the the support condition of $\xi$ and $\eta$ follows from $|x|,|x+\xi|,|x+\eta|\le 2$.
To show (\ref{081}), use the change of variable $\xi\rightarrow \xi+\eta$  in  (\ref{813}) to have
\begin{align*} 
\int \left| \psi\left(\frac{ \mathcal{P}_{\mathcal{E}^{\perp}(\mu)}(\xi)  }{  \frac{1}{\rho} 2^{(2\epsilon-1)j}    }\right)K^{\mu}_{\rm{main}}(\xi+\eta,\eta) \right|d\xi  \le \int \left| \psi\left(\frac{ \mathcal{P}_{\mathcal{E}^{\perp}(\mu)}(\xi)  }{  \frac{1}{\rho} 2^{(2\epsilon-1)j}    }\right)\frac{ 2^{jd}  \psi(\xi+\eta)\psi(\eta) }{ 2^{(1- \epsilon)j} |\xi| +1 } \right|d\xi.
 \end{align*}
 By using (\ref{215a}), we majorize the RHS above by
  \begin{align*}
&\int_{\xi_b\in \mathcal{E}^{\perp}(\mu)}   \psi\left(\frac{\xi_b  }{  \frac{2^{j(2\epsilon-1)}}{\rho}     }\right) \left[ \sup_{\xi_b\in \mathcal{E}^{\perp}(\mu)}\int_{\xi_a\in \mathcal{E}(\mu)}\frac{ 2^{jd} \psi(\xi_a+\xi_b+\eta)\psi(\eta)}{2^{j(1- \epsilon)} |\xi_a+\xi_b| +1 } d\xi_{a} \right] d\xi_b\nonumber\\
&\quad \lesssim \left( 2^{j(2\epsilon-1)} /\rho\right)^{d-m_g(\mu)}   [2^{jd }2^{j\epsilon}2^{j(\epsilon-1)}]\lesssim 2^{jc\epsilon }2^{j(m_g(\mu)-1)} 
  \end{align*} 
to show (\ref{081}). \end{proof}
 Hence,  to prove    (\ref{cx1}), there remains to show that
\begin{proposition}
Let $\mu\in \Lambda(A)\cap \mathbb{R}$ and $m_a(\mu)=m_g(\mu)$.   Then  we have
\begin{align}\label{814}
&\int_{B(\eta)^c} \left| K^{\mu}_{\rm{main}}(\xi,\eta) \right|d\xi \lesssim  1 \ \text{if} \  B(\eta)^c=\left\{\xi:   \frac{2^{j(2\epsilon-1)}}{\rho}<|\mathcal{P}_{\mathcal{E}^{\perp}(\mu)}(\xi-\eta)|   \right\}. \end{align} 
\end{proposition}
To show (\ref{814}),  choose the best direction $v $ for the   derivatives of the phase:
\begin{lemma}
 For $(x,t,\xi,\eta)$ in  the support of the above integral   in (\ref{pc98}),
   there exists a unit vector $v\in \mathcal{E}^{\perp}(\mu)$   independent of $x$ such that
\begin{align}
&|\langle v, \nabla_x\rangle \Phi(x,t,\xi,\eta)|\approx  |\langle v, (A-\mu I)(\xi-\eta)\rangle|\approx |\mathcal{P}_{\mathcal{E}^{\perp}(\mu)}(\xi-\eta)|\ \text{and} \label{739}\\ 
&\left|\langle  v, \nabla_x\rangle^m\left(   \langle v,\nabla_x\Phi(x,t,\xi,\eta)\rangle\right)\right| \le  |\mathcal{P}_{\mathcal{E}^{\perp}(\mu)}(\xi-\eta)|\ \text{for any $m\ge 1$}.\label{743}
\end{align}
Moreover, there exists an upper bound of the derivatives of the amplitude 
\begin{align}\label{0099}
 \langle v, \nabla_x \rangle^m \left[\psi\left(\frac{ |\xi+x| -|\eta+x|
  }{2^{j(\epsilon-1)}}\right) \right]=O\left(  \left[\frac{2^j}{2^{j\epsilon}} |\mathcal{P}_{\mathcal{E}^{\perp}(\mu)}(\xi-\eta)|+1\right]^m \right). 
  \end{align}
\end{lemma}
 \begin{proof}[Proof of (\ref{739})]
From   $|\xi-\eta|\ge 2^{(3 \epsilon-1) j} $  and
 $ \left| |\xi+x| -|\eta+x|\right| \le 2^{(\epsilon-1)j} $
in (\ref{pc98}),  
   \begin{align} \label{409}  
 \nabla_x \Phi(x,t,\xi,\eta)
 &= A (\xi-\eta) +t\left( \frac{\xi-\eta}{|\xi+x|}  +(\eta+x)  \left( \frac{1}{|\xi+x|} - \frac{1}{|\eta+x|} \right) \right) \nonumber \\
&= \left(A-\mu I\right) (\xi-\eta)+ \left(\mu+\frac{t}{|\xi+x|}  \right)   (\xi-\eta)  +O\left(2^{(\epsilon-1)j}\right). 
   \end{align} 
From (\ref{pv61}) and (\ref{59}),  by choosing $$v:=\frac{\mathcal{P}_{\mathcal{E}^{\perp}(\mu)} \left(\left(A-\mu I\right) (\xi-\eta)\right)  }{|  \mathcal{P}_{\mathcal{E}^{\perp}(\mu)}\left(\left(A-\mu I\right) (\xi-\eta)\right)   |} \in \mathcal{E}(\mu)^{\perp},$$   one can obtain  that
\begin{align}\label{740}
|\langle v,\left(A-\mu\right) (\xi-\eta)  \rangle| \approx|  \left(A-\mu I\right) (\xi-\eta)   |\approx  |\mathcal{P}_{\mathcal{E}^{\perp}(\mu)}(\xi-\eta)|.
\end{align}
From $ \left\langle v,  \mathcal{P}_{\mathcal{E}(\mu)}(\xi-\eta)\right\rangle=0$ in  (\ref{409})  due to $v\in \mathcal{E}(\mu)^{\perp}$, it follows  that
\begin{align}
 \langle v,\nabla_x \Phi(x,t,\xi,\eta)\rangle&=\langle v,\left(A-\mu I\right) (\xi-\eta)  \rangle+ \left\langle v, \left(\mu+\frac{t}{|\xi+x|}  \right)  \mathcal{P}_{\mathcal{E}^{\perp}(\mu)}(\xi-\eta)\right\rangle\nonumber\\
 &+ t \left( \frac{1}{|\xi+x|} - \frac{1}{|\eta+x|} \right)  \langle v, (\eta+x)  \rangle.  \label{pv53}
\end{align}
Then this together with the condition (1)-(2) below
 \begin{itemize}
\item[(1)]  $\left|\mu+\frac{t}{|\xi+x|}\right|\le \rho\ll 1$   in the support of $\psi\left(\left( \mu  +  \frac{t}{|\xi+x|}\right)\frac{1}{\rho}\right)$ 
\item[(2)]
$| \mathcal{P}_{E^{\perp}}(\mu)(\xi-\eta)|\gtrsim 2^{j(2\epsilon-1)} $ in (\ref{814})
\end{itemize}
   implies that the first term of the RHS in (\ref{pv53}) has a size $\approx |\mathcal{P}_{\mathcal{E}^{\perp}(\mu)}(\xi-\eta)|$    in (\ref{740}) while the second term $\ll  |\mathcal{P}_{\mathcal{E}^{\perp}(\mu)}(\xi-\eta)|$ and  the third term $=O(2^{j(\epsilon-1)})$.   
\end{proof}
\begin{proof}[Proof of (\ref{743})]
In the above, we proved the case $m=1$. To show $m= 2,3,\cdots,$ observe that  $\nabla_x$ of the first two terms of (\ref{pv53}) is controlled by $ |\mathcal{P}_{\mathcal{E}^{\perp}(\mu)}(\xi-\eta)|$, and the third term controlled by
\begin{align*} 
 \langle v,\nabla_x \rangle \left(\frac{1}{|\xi+x|}-\frac{1}{|\eta+x|}\right) &=\langle v,\xi+x\rangle \left(\frac{1}{|x+\xi|^3}-\frac{1}{|x+\eta|^3}\right)+\frac{\langle v,\xi-\eta\rangle}{|\eta+x|^3}.
 \end{align*}
Then the first term is controlled by $ \big| |\xi+x|-|\eta+x| \big| =O(2^{j(\epsilon-1)})$ and the second     $\frac{\langle v,\xi-\eta\rangle}{|\eta+x|^3} \approx |P_{\mathcal{E}^{\perp}(\mu)}(\xi-\eta)| $ as $v\in \mathcal{E}^{\perp}(\mu)$. Repeat this to obtain (\ref{743}). \end{proof}
 \begin{proof}[Proof of (\ref{0099})]
In the similar way as   in (\ref{409}) with $| |\xi+x| -|\eta+x||\le2^{j(\epsilon-1)}$, \begin{align*} 
\left| \langle v, \nabla_x \rangle \left[\psi\left(\frac{ |\xi+x| -|\eta+x|
  }{2^{j(\epsilon-1)}}\right)\right] \right|&\lesssim \left| \frac{1}{2^{j(\epsilon-1)}}\left\langle v, \frac{\xi-\eta}{|\xi+x|} \right\rangle+O\left(1\right) \right| 
   \end{align*}  
which is $O(2^{j(1-\epsilon)} |\mathcal{P}_{\mathcal{E}^{\perp}(\mu)}(\xi-\eta)|+1 )$.   This yields (\ref{0099}).
 \end{proof}

  \begin{proof}[Proof of (\ref{814})]
Take $v$ as above.  Put
$$\mathcal{L}[g(x,t,\xi,\eta)]=\langle v, \nabla_x \rangle \left[ \frac{g(x,t,\xi,\eta)}{   i 2^j\langle v, \nabla_x \rangle  [\Phi(x,t,\xi,\eta)] }\right] 
.$$
Apply the integration by parts with respect to $dx$ along the direction $v$. Then
   \begin{align*} 
|K^{\mu}_{\rm{main}}(\xi,\eta)|&= 2^{jd}\bigg|\int_{\mathbb{R}^{d+1} } e^{ i 2^j 
\Phi(x,t,\xi,\eta)}
\mathcal{L}^N\bigg[
 \psi(x)\chi(t)\chi\left( 
t |\xi+ x| \right)\chi\left(  t |\eta+
x| \right)\nonumber\\
&
\times \psi\left(\left( \mu  +  \frac{t}{|\xi+x|}\right)\frac{1}{\rho}\right) \psi\left(\frac{ |\xi+x| -|\eta+x|
  }{2^{j(\epsilon-1)} }\right) \bigg]dxdt\bigg|.
  \end{align*}
  We utilize this   together with (\ref{739})-(\ref{0099})   to obtain that
  \begin{align*}
&|K^{\mu}_{\rm{main}}(\xi,\eta)|  \lesssim     \frac{\left[\frac{2^j}{2^{j\epsilon}} |\mathcal{P}_{\mathcal{E}^{\perp}(\mu)}(\xi-\eta)|+1\right]^N  \tilde{\psi}(\xi)\tilde{\psi}(\eta)}{|2^j\mathcal{P}_{\mathcal{E}^{\perp}(\mu)}(\xi-\eta)|+1|^N}   \lesssim  
  2^{jd}  \frac{   \tilde{\psi}(\xi)\tilde{\psi}(\eta)}{2^{\epsilon N j}}.
 \end{align*}
This with the support condition $ | \xi|,|\eta|  <5$ yields that for $\epsilon N\ge d$,
  \begin{align*}
\int_{B(\eta)^c} |K^{\mu}_{\rm{main}}(\xi,\eta)|d\xi \lesssim  2^{j (d-\epsilon N)}\le  1 .
  \end{align*}
Therefore we obtain (\ref{814}).
  \end{proof}

\section{Lower Bounds in Main Theorem 2}\label{sec9}
 \begin{lemma}
 For    $A\in G(\mathbb{R}^d)$, let $\mu\in\Lambda(A)\cap \mathbb{R}$.
  Suppose that  $\text{dim}(\mathcal{E}(\mu))=m$. 
  Then it holds that
\begin{align}\label{nv32}
  &\|\mathcal{M}_{S^{d-1}(A)} \|_{L^p(\mathbb{R}^{d+1})\rightarrow L^p(\mathbb{R}^{d+1})}=\infty
 \ \text{$\forall p\in \left[1,\frac{d-(m-1)}{d-(m-1)-1}\right]$ if $m<d$,}\\
&|\mathcal{M}_{S^{d-1}(I)} \|_{L^p(\mathbb{R}^{d+1})\rightarrow L^p(\mathbb{R}^{d+1})}=\infty\ \text{$\forall p\in [1,\infty)$ if $m=d$}\nonumber
\end{align}
\end{lemma}
\begin{example}
The case  $m=d$  occurs at $A=cI$ for $c\ne 0$ where  $I$ is the identity. 
\end{example}
\begin{proof}[Proof of (\ref{nv32})]
To prove (\ref{nv32}), we first claim the unboundedness of the operator norm for all $p<(d-m+1)/(d-m)$. Next, we   show $|\mathcal{M}_{S^{d-1}(A)}f(x,x_{d+1})|=\infty$ on  $(x,x_{d+1})\in B$ for $|B|\approx 1$ and   $f\in L^p(\mathbb{R}^{d+1})$ with  $p=(d-m+1)/(d-m)$ by taking $t=t(x,x_3)$ below,
\begin{align}\label{ak90}
 \mathcal{A}_{S^{d-1}(A)}f\left(x,x_{d+1},t \right)&= \int_{ S^{d-1} } f \left(x-ty,\, x_{d+1} -\langle x,  tAy\rangle \right)d\sigma(y).
\end{align}
Let $\{q_1,\cdots,q_m\} $ be an orthonormal basis of $\mathcal{E}(\mu)$ and let $\{q_1,\cdots,q_m,q_{m+1},\cdots,q_d\}$
be an orthonormal basis  of $\mathbb{R}^d$. Since $A$ is invertible, it holds that $\mu\ne 0$ and $\{q_1,\cdots,q_m,\tilde{q}_{m+1},\cdots,\tilde{q}_d\}$ are linearly independent where  
$$  Aq_1=\mu  q_1,\cdots,Aq_m=\mu q_m, Aq_{m+1}= \tilde{q}_{m+1},\cdots, Aq_{d}= \tilde{q}_{d}.$$
Consider the  two matrices:
\begin{itemize}
\item $Q=(q_1|\cdots|q_d)$, which is an orthogonal matrix and 
\item $\tilde{Q}=(q_1|\cdots|q_m|\tilde{q}_{m+1}|\cdots|\tilde{q}_d)$, which is an invertible matrix. 
\end{itemize}
We can observe that
 $AQ=\tilde{Q}D$ where $D$ is the diagonal matrix whose entries are $\lambda,\cdots,\lambda,1,\cdots,1$. Since $Q\in SO_d(\mathbb{R})$, it suffices to switch $A$ with  $Q^TAQ$ in our proof.  
Note that $$(Q^T\tilde{Q})=(\langle q_i,\tilde{q}_j\rangle)\ \text{ where $ \langle q_i,\tilde{q}_j\rangle=\delta_{ij}$ for $j\le m $.}$$
Hence
\begin{align*}
\langle x, Q^TAQy\rangle&= \langle Qx, \tilde{Q} D y\rangle=\left\langle \sum_{i=1}^d x_iq_i, \sum_{j=1}^m \mu y_jq_j+    \sum_{j=m+1}^dy_j\tilde{q}_j   \right\rangle\\
&=
\mu\sum_{i=1}^m x_iy_i+ \sum_{j=m+1}^d\sum_{i=1}^d \langle q_i,\tilde{q}_j\rangle    x_iy_j
\end{align*} 
which can be rearranged as
\begin{align*}
 &\left(\mu\sum_{i=1}^d x_iy_i\right)-  \left(\mu\sum_{j=m+1}^d x_j(y_j-x_j)  + \sum_{j=m+1}^d\sum_{i=1}^d \langle q_i,\tilde{q}_j\rangle  x_i(y_j-x_j) \right) \\
 &-\left(\mu\sum_{j=m+1}^d x_j^2- \sum_{j=m+1}^d\sum_{i=1}^d \langle q_i,\tilde{q}_j\rangle    x_i x_j\right)=\mu\langle x,y\rangle -\sum_{j=m+1}^d B_j(x)  (y_j-x_j)-R(x). 
 \end{align*}
 where
 $B_j(x)=\mu x_j-\sum_{i=1}^d \langle q_i,\tilde{q}_j\rangle  x_i$ and $R(x)=\sum_{j=m+1}^d x_j  B_j(x)$. So, by
taking $$[f](x,x_{d+1})=f(x,x_{d+1}+\frac{\mu}{2}|x|^2)$$  one  can rewrite (\ref{ak90}) as
\begin{align*}
[ \mathcal{A}_{S^{d-1}(A)}f] \left(x,x_{d+1},t \right)&= \int_{ S^{d-1} }f\left(x-ty,\, x_{d+1}+\frac{\mu}{2} |x|^2-\langle x,  Q^TAQ ty\rangle \right)d\sigma(y)\\
&=  \int_{S^{d-1} } [f]\left(x-ty,\, x_{d+1}+ \mu\langle x, ty\rangle  
-t^2\frac{\mu}{2}  - \langle x,  Q^TAQ ty\rangle \right)d\sigma(y) 
\end{align*}
where  the $(d+1)^{th}$ component of $[f]$ is expressed as
$$ \left(x_{d+1}+ \mu\langle x, ty\rangle  
-t^2\frac{\mu}{2} \right)+\left( -  \mu\langle x,ty\rangle-\sum_{j=m+1}^d B_j(x)(x_j-ty_j)+R(x) \right). $$
Next use the change of variable $x_{d+1}+R(x)\rightarrow x_{d+1}$. We have
\begin{align}\label{oop}
 &[\mathcal{A}_{S^{d-1}(A)}f]\left(x,x_{d+1},t \right)\nonumber\\
 &\qquad=  \int_{S^{d-1}}[f]\left(x-ty,\, x_{d+1}   
-t^2\frac{\mu}{2}  +\sum_{j=m+1}^d B_j(x)(ty_j-x_j)    \right)d\sigma(y).  
\end{align}
Let  $\mu\approx 1$  without loss of generality and  set $$B=\{(x,x_{d+1}):|x|\le 100\ \text{and}\  1\le |x_{d+1}|\le 2\}.$$  In (\ref{oop}), we take
\begin{align}\label{a599}
[f](y,y_{d+1})=\psi\left(\frac{y_1}{C},\cdots,\frac{y_{m}}{C}\right)\psi\left(\frac{y_{m+1}}{\delta},\cdots,\frac{y_{d}}{\delta}\right)\psi\left(\frac{y_{d+1}}{C\delta}\right)
\end{align}
where $C\gg 1$ and choose 
\begin{align*} 
t=t(x_{d+1}):= \sqrt{2|x_{d+1}|/\mu } \ \text{ satisfying}\  x_{d+1}-\frac{\mu}{2} t^2=0\ \text{in (\ref{oop})}.
 \end{align*}
If  $|x_j-ty_j|<\delta$ for $j=m+1,\cdots,d$,  then in the last component of $[f]$ in (\ref{oop}),  
\begin{align}\label{04nm}
\left|\sum_{j=m+1}^d B_j(x)(ty_j-x_j) \right|\le C\sum_{j=m+1}^d |ty_j-x_j|=O(\delta)\ \text{ for $(x,x_{d+1})\in B$.}
\end{align} 
Thus  for  $(x,x_{d+1})\in B$ and  $t=t(x_{d+1})$ with $f$ in (\ref{a599}), we can write (\ref{oop}) as
\begin{align}\label{ak50}
&[\mathcal{A}_{S^{d-1}(A)}f]\left(x,x_{d+1},t(x_{d+1})\right)=\nonumber\\
&\quad \int_{  S^{d-1} } \psi\left(\frac{x_1-ty_1}{C},\cdots,\frac{x_m-ty_m}{C}\right)\psi\left(\frac{(x_{m+1}-ty_{m+1})}{\delta},\cdots,\frac{(x_{d}-ty_d)}{\delta} \right) d\sigma(y). 
\end{align}
We can regard $ y_{m+1},\cdots,  y_{d}$ as free variables in the restriction
$$y_1=\sqrt{1-(y_2^2+\cdots+ y_m^2+ y_{m+1}^2+\cdots+y_d^2)}  $$
 and use $t =\sqrt{2|x_{d+1}|/\mu }\approx \sqrt{2|x_{d+1}| }$ to compute the above integral as
\begin{align*} 
\left|[\mathcal{A}_{S^{d-1}(A)}f]\left(x,x_{d+1},t(x,x_{d+1}) \right)\right|\approx \left(\frac{\delta}{\sqrt{|x_{d+1}|}}\right)^{d-m}\approx \delta^{d-m} 
\end{align*}
for each $(x,x_{d+1})\in B $.
Hence  it holds that
 \begin{eqnarray*} 
 \|\mathcal{M}_{S^{d-1}(A)} f \|_{L^p(\mathbb{R}^{d+1})}^p\ge \int_{B} \left|[\mathcal{A}_{S^{d-1}(A)}f](x,x_{d+1},t(x_{d+1}))\right|^p
dxdx_{d+1}\ge c \delta^{(d-m)p} .
\end{eqnarray*}
From (\ref{a599}), we have $\| f \|^p_{L^p(\mathbb{R}^{d+1})}\le  \delta^{d-m+1}.
$
This   implies that  $$
\|\mathcal{M}_{S^{d-1}(A)}\|^p_{L^p(\mathbb{R}^{d+1})\rightarrow
L^p(\mathbb{R}^{d+1})}\gtrsim\frac{\delta^{(d-m)p}}{\delta^{d-m+1}}
$$ which goes to $\infty$ if $\delta\rightarrow 0$   for $p<\frac{d-m+1}{d-m}$ in (\ref{nv32}). Let $p= \frac{d-m+1}{d-m}$. 
Fix $\delta>0$ and replace $  [f](y,y_{d+1})$ in (\ref{a599})  with
\begin{align*} 
\psi\left(\frac{y_1}{C},\cdots,\frac{y_{m}}{C}\right)\psi\left(\frac{y_{m+1}}{\delta},\cdots,\frac{y_{d}}{\delta}\right)\psi\left(\frac{y_{d+1}}{C\delta}\right)\frac{\log (|(y_{m+1},\cdots,y_d,y_{d+1})|)^{-1}  }{|(y_{m+1},\cdots,y_d,y_{d+1})|^{d-m}}.
\end{align*}
Then $\| f \|^p_{L^p(\mathbb{R}^{d+1})}$ is bounded by
$$ \int \frac{|\log (|(y_{m+1},\cdots,y_d,y_{d+1})|)^{-1} |^p }{|(y_{m+1},\cdots,y_d,y_{d+1})|^{d-m+1}}\psi\left(\frac{y_{m+1}}{\delta},\cdots,\frac{y_{d}}{\delta}\right)\psi\left(\frac{y_{d+1}}{C\delta}\right)dy_{m+1}\cdots dy_ddy_{d+1} 
$$
which is  also bounded  $\int_{r=0}^{r=C\delta} \frac{|\log  r|^{-p}}{r}dr<\infty$ due to $p>1$.
On the other hand, the corresponding  $[\mathcal{A}_{S^{d-1}(A)}f]\left(x,x_{d+1},t(x_{d+1})\right)$  for each $(x,x_{d+1})\in B $ with $t=t(x_{d+1})$ in (\ref{ak50})  is bounded away from  
\begin{align*} 
  & \int  \psi\left( \frac{(x_{m+1}-ty_{m+1})}{\delta},\cdots,\frac{(x_{d}-ty_d)}{\delta} \right) \frac{|\log (|(x_{m+1}-ty_{m+1}),\cdots,(x_{d}-ty_{d}))|)|^{-1}  }{|(x_{m+1}-ty_{m+1},\cdots,x_{d}-ty_{d})|^{d-m}} \\
&  dy_{m+1}\cdots dy_d  
   =\int  \psi\left( \frac{y_{m+1}}{\delta},\cdots,\frac{y_d}{\delta} \right) \frac{|\log (|(y_{m+1},\cdots, y_{d})|)|^{-1}  }{|(y_{m+1},\cdots, y_{d})|^{d-m}} dy_{m+1}\cdots dy_d\\
   &\approx \int_{r=0}^{r= \delta}\frac{1}{r}|\log(r)|^{-1}dr=\infty.
\end{align*}
  This shows   $[\mathcal{M}_{S^{d-1}(A)}f]\left(x,x_{d+1})\right)=\infty$ for $(x,x_{d+1})\in B$, which leads (\ref{nv32}). 
\end{proof}

\section{$L^2$ Sobolev Regularity, Proof of Main Theorem \ref{main4}}

\subsection{Ranks of Mixed Hessian Matrices }
Let $\phi(x,t,\xi)=\langle Ax,\xi\rangle +  t |x+\xi|$. We shall compute  ranks and determinants  of $[\phi]_{x\xi}''(x,t,\xi)$  according to various $A$.
\begin{lemma}\label{lem25}
Let $A\in G_{d}(\mathbb{R})$. Then all diagonal entries of $Q^TAQ$ for every $Q\in SO(\mathbb{R}^d)$ are zeros if and only if $A$ is a skew symmetric matrix.
\end{lemma}
\begin{proof}
We see that $\Leftarrow$ follows from the fact that $Q^TAQ$ is skew symmetric whenever $A$ is skew symmetric. We shall show $\Rightarrow$. Let $A=(a_{ij})$. Suppose that all diagonal entries   of $Q^TAQ $ are zeros  for  all $Q\in O(\mathbb{R}^d)$. Then one can see that
  $a_{ii}=0$ for all $i\in [d]$ by taking   $Q=I$.
Let $q:=\frac{1}{\sqrt{2}}({\bf e}_i+{\bf e}_j)$ be the $k^{th}$ column vector of some $Q\in SO(\mathbb{R}^d)$. Then  the $k^{th}$ diagonal entry of $Q^TAQ$ is $0=  \langle q, Aq\rangle =\frac{1}{2}\langle ({\bf e}_i+{\bf e}_j),(a_{\ell i}+a_{\ell j})_{\ell}\rangle=\frac{1}{2}(a_{i i}+a_{i j}+a_{j i}+a_{j j})=a_{ij}+a_{ji}$. Hence $A$ is skew symmetric.
\end{proof}
\begin{lemma}\label{lem515}
Let $A\in G_{d}(\mathbb{R})$  such that  $\Lambda(A)\cap \mathbb{R}=\emptyset$ and let $\phi(x,t,\xi)=\langle Ax,\xi\rangle +  t |x+\xi|$.   
 If $A$ is  skew symmetric, then   for all $t$ and $1/2\le |x+\xi|\le 2$,
\begin{align}\label{449}
|\det( \phi''_{x\xi}(x,t,\xi))|\ge c(1+|t|^{d-2}) 
\end{align}
with $c$ independent of $x,\xi,t$. If $A$ is not   skew symmetric and $t\ne 0$,  there is $x,\xi$: 
\begin{align}\label{s7}
 |\det( \phi''_{x\xi}(x,t,\xi))|=0\ \text{where $  |x+\xi|\approx t$.}
\end{align}
\end{lemma}
\noindent
{\bf Observation}.\ \ Put $u=(u_i):=(x_i+\xi_i)$ with $|u|\approx 1$ and express
\begin{align}\label{451}
 [\phi]_{(x_1\cdots x_d)(\xi_1\cdots \xi_d)}''(x,t,\xi)&=    A+\frac{t}{|u|}I -\frac{t}{|u|^3}U\ \text{where}\ U=\left(\begin{matrix}
 u_1 u\\
 u_2 u\\
  \vdots\\
u_d u\\
 \end{matrix}  \right). 
\end{align}
 If $t=0$, one has (\ref{449}) as  $A\in G_d(\mathbb{R})$.  Let $t\ne 0$ and  $k\in [d]$.   In  (\ref{451}),  
observe that $U$ is a symmetric matrix of rank $1$, and  its eigenvalues are $0,\cdots,0,|u|^2,0,\cdots,0$ with their eigenvectors  $q_k=\frac{u}{|u|}$ and $q_i\in \text{span}^{\perp}(q_k)$ for $i\in [d]\setminus\{k\}$. So
\begin{align}\label{gamu1}
  Uq_k=|u|^2q_k \ \text{and }\  Uq_i=0q_i\ \text{for all $i=[d]\setminus\{k\}$} 
\end{align}
and they form  $Q=(q_1|\cdots|q_d)\in SO_d(\mathbb{R})$ such that $U=QDQ^T$ and   $D $ is a diagonal matrix $|u|^2E_k$  with $E_k=({\bf 0}|\cdots|{\bf 0}|{\bf e}_k|{\bf 0}|\cdots|{\bf 0})$. 
Express RHS in (\ref{451}) as
\begin{align} \label{gamu2}
Q(Q^TAQ+\frac{t}{|u|}I -\frac{t}{|u|^3}D)Q^T=Q(Q^TAQ+\frac{t}{|u|}\hat{I}_k)Q^T.
\end{align}
where $\hat{I}_k=I-E_k$.
 Then we have
\begin{align} \label{s5}
\det( [\phi]_{x\xi}''(x,t,\xi))=\det(Q^TAQ+\frac{  t}{|u|}\hat{I}_k).
\end{align}
\begin{proof}[Proof of (\ref{449})]
Put  $\mu:=t/|u|$ in (\ref{s5}). Since $|u|\approx 1$,   it suffices to show that  
\begin{align}\label{450}
|\det(Q^TAQ+\mu\hat{I}_k)|\ge c(1+|\mu|^{d-2})\ \text{ for   $c>0$ depending only on $A$.}
\end{align}
{\bf Case 1}. If $|\mu|\ll 1$,   $|\det(Q^TAQ+\mu\hat{I}_k)|=|\det(Q^TAQ)|+O(\mu)\approx |\det(A)|>0$. \\
{\bf Case 2}. Let $|\mu|\gg 1$. Fix $A$ and set the continuous mapping
$Q\in SO_d(\mathbb{R})\rightarrow  |Q^TAQ({\bf e}_k)|\in \mathbb{R}$ where $Q^TAQ({\bf e}_k)$ is the $k^{th}$ column of $Q^TAQ$. Then from compactness of $SO_d(\mathbb{R})$ and $\det(Q^TAQ)=\det(A)\ne 0$, one see that
$$\min_{Q\in SO_d(\mathbb{R}) } |Q^TAQ({\bf e}_k)|\ge c>0.$$
Put
  $Q^TAQ=(a_{ij}(q))$. Then we write $\det(Q^TAQ+\mu \hat{I}_k)$ as
\begin{align*}
a_{kk}(q)\mu^{d-1}-(a_{1k}(q)a_{k1}(q)+\cdots+a_{dk}(q)a_{kd}(q))\mu^{d-2}+\cdots+ \det(A)\mu^0.
\end{align*}
where $a_{kk}(q)=0$ and  $a_{ki}(q)=-a_{ik}(q)$ because $ Q^TAQ$  is   skew symmetric.   Thus
\begin{align*}
&\det(Q^TAQ+\mu \hat{I}_k)  \approx |Q^TAQ({\bf e}_k)|^2\mu^{d-2}
\end{align*}
as   $|\mu|\gg 1$. This shows (\ref{450}) for $|\mu|\gg 1$.\\
{\bf Case 3}. Let $|\mu|\approx 1$. It suffices to work with $k=d$ in (\ref{450}) and write  $Q^TAQ$ as $[A]$ to claim that $|[A]x+\mu \hat{I}_dx|\ge c|x|$ for all $x$. 
Note that $\hat{I}_dx=(x_1,\cdots,x_{d-1},0)^T$.   Consider the following two cases.\\
{\bf Case 3-1}. $|\hat{I}_dx|\approx |x|$. Then $|(x_1,\cdots,x_{d-1})|\approx |x|$, by using  $\langle [A]x, x\rangle=0$.
$$| [A]x+\mu\hat{I}_dx|\ge |\langle [A]x+\mu\hat{I}_dx, \frac{x}{|x|}\rangle|=|\mu\frac{(x_1^2+\cdots+x_{d-1}^2)}{|x|}|\approx |x|.$$
{\bf Case 3-2}.  $|\hat{I}_dx|\ll |x|$.  Then, from $|[A]x|\gtrsim |x|$ due to $[A]\in G_d(\mathbb{R})$,
$$|  [A]x+\mu\hat{I}_dx|\ge |[A]x|-|\mu  \hat{I}_dx|\gtrsim |x| -  |\mu\hat{I}_dx| \approx |x|.$$
showing (\ref{450}) for $|\mu|\approx 1$.  Therefore,  from  ${\bf Cases}$ 1-3, we obtain (\ref{449}).
\end{proof}
\begin{proof}[Proof of (\ref{s7})]
Suppose that $A$ is not skew symmetric. Then by Lemma \ref{lem25},  one can see that 
$Q^TAQ=(s_{ij})$   has a non-zero diagonal entry   $s_{ii}\ne 0$.  Without loss of generality, let $i=1$ so that $s_{11}\ne 0$. In (\ref{451})-(\ref{s5}),  one can  take $k=1$ and put  $\mu:=t/|u|$.   
For $s_{11}\ne 0$ above, we have the $d-1$ degree polynomial in $t$ such as
$$ \det(Q^TAQ+\mu\hat{I}_1 )=s_{11}\mu^{d-1}+()\mu^{d-2}+\cdots+()\mu^1+\det(A).$$
On the other hand,  the hypothesis $\Lambda(A)\cap \mathbb{R}= \emptyset$ implies that  $d$ is even (a characteristic polynomial of odd degree has at least one real solution). 
So,  $d-1$ is odd with  $\det(A)\ne 0$ implies that there exists a zero $\mu=\mu_0\ne 0$ of the above odd degree polynomial $ \det(Q^TAQ+\mu\hat{I}_1 )$. Hence, for $t\ne 0$, we can choose $x ,\xi $ such that  $ \frac{t}{|x+\xi|}=\mu_0$ which  satisfies (\ref{s7}).  
\end{proof}

\begin{lemma}\label{lem44}
Let  $\phi(x,t,\xi)=\langle Ax,\xi\rangle +  t |x+\xi|$ and $A\in G_d(\mathbb{R})$.  If $\mu\in \Lambda(A)\cap \mathbb{R}$ with  $m=\text{dim}(\mathcal{E}(\mu))$,
then given $t\ne0$, there are    $ x,\xi$ with $|x+\xi|\approx |t|$:
\begin{align}\label{c44}
\text{rank}(\phi''_{x\xi}(x,t,\xi))\le  d-m.
\end{align}

\end{lemma}
\begin{proof}[Proof of (\ref{c44})]
Let $\mu\in\Lambda(A)\cap \mathbb{R}$. Then $\mu\in \Lambda(Q^TAQ)$  for all $Q\in SO_d(\mathbb{R})$ and    
\begin{align}\label{c48}
Q^T\mathcal{E}^{\perp}(\mu)=[Q^T\mathcal{E}(\mu)]^{\perp}=\mathcal{R}(Q^TAQ-\mu I)\ \text{as $\mathcal{R}$ indicates row space}
\end{align}
where the first equality follows from $\langle Q^Tw, Q^Tv\rangle =0$ for $Q^T\in SO_d(\mathbb{R})$ with $w\in \mathcal{E}^{\perp}(\mu)$ and $v\in  \mathcal{E}(\mu) $.  The   the second follows from  
  $(Q^TAQ-\mu I) Q^Tv=Q^T(A-\mu I)v={\bf 0}$ for $v\in\mathcal{E}(\mu)$.
Choose
\begin{align}\label{c50}
Q\in SO_d(\mathbb{R})\ \text{satisfying $q_1=Q{\bf e}_1\in  \mathcal{E}^{\perp}(\mu)$ }.
\end{align}
In such  $Q =(q_1|q_2|\cdots|q_d)$, take $u=(x+\xi)\parallel q_1$ to construct
  a  symmetric   matrix  $U$ satisfying (\ref{gamu1}) and (\ref{gamu2}) with $k=1$.  Take $\mu=-\frac{t}{|u|}$. 
Then in (\ref{gamu2}), 
\begin{align*}
 Q(Q^TAQ+\frac{t}{|u|} I- \frac{t}{|u|}E_1  )Q^T= Q(Q^TAQ-\mu I+\mu E_1  )Q^T. 
\end{align*}
By  (\ref{c50}), the first row of $\mu E_1$ is
  $\mu {\bf e}_1\in Q^T\mathcal{E}^{\perp}(\mu)=\mathcal{R}(Q^TAQ-\mu I)$ from (\ref{c48}). By this  
we have $\text{rank}(\phi''_{x\xi}(x,t,\xi))=\text{rank} (Q^TAQ-\mu I +\mu E_1)\le \text{rank} (Q^TAQ-\mu I )=\text{rank} (  A -\mu I )=d-m$.    \end{proof}
 
\begin{lemma}
For $v \in S^{d-1}$, let $\text{cone} (v):=\{tv :t\ge 0\}$ and  define the projection  $x \in\mathbb{R}^d$ to   the   ray  $\text{cone}(v)$ as
$P_{v}(x):=(x\cdot v )v$. 
Then we have
\begin{align}\label{0441}
\big| |x |-|P_v(x)|\big|\le \frac{|x-P_{v}(x) |^2 }{|x|}.
\end{align}
Suppose that $x,\xi\in \rm{cone}(v )+O(\eta)$  with $x\cdot v>0$ and $\xi\cdot v >0$. Then we have
\begin{align}\label{0442}
|x+\xi|=(x+\xi)\cdot v +O\left(\frac{\eta^2}{|x+\xi|}\right).
\end{align}
\end{lemma}
\begin{proof}
  Let $\theta$ be the angle between $x$ and its projection $P_v(x)=\langle x,v\rangle v$. Then, from the right triangle having vertices ${\bf 0},x, P_{v}(x)$, observe  that
$$\text{$\sin \theta=\frac{|x- P_{v}(x)|}{|x|} $ and}\  \big| |x |-|P_v(x)|\big|\le |x|(\sin\theta)^2.$$
This implies (\ref{0441}).  
Next, from the hypothesis,  one has $|P_v(x+\xi)|=(x+\xi)\cdot v$ and $x+\xi\in  \text{cone}(v)+O(2\eta)$, that is $|x+\xi-P_{v}(x+\xi)|<2\eta$.  Therefore, one can insert $x+\xi$ into $x$ in (\ref{0441}) to obtain (\ref{0442}).
\end{proof}
 \subsection{$L^2(dx dt)$ estimate} 
We shall prove  (\ref{0045}).  
Let  $\alpha\ge 0$. In order to have  \begin{align*} 
 \|\mathcal{A}_{S^{d-1}(A)}f\|_{L^2_{\alpha}(\mathbb{R}^{d+1}\times [1,2])} & \approx   \|[\partial^{\alpha}\mathcal{A}]f\|_{L^2(\mathbb{R}^{d+1}\times [1,2])}
\end{align*}
we assign the operator $f\rightarrow [\partial^{\alpha}\mathcal{A}]f$ to each $\alpha$, whose value at  $ (x,x_{d+1},t)$ is  
\begin{align*}
 &\int e^{2\pi i \langle (x,x_{d+1}),(\xi,\xi_{d+1})\rangle} m^{\alpha}(x,x_{d+1},t,\xi,\xi_{d+1})  \widehat{f}(\xi,\xi_{d+1})d\xi d\xi_{d+1}\ \text{with}\\
& m^{\alpha}(x,x_{d+1},t,\xi,\xi_{d+1}) =(|\xi|+|\xi_{d+1}|)^{\alpha}\widehat{d\sigma}(2\pi |\xi+\xi_{d+1}Ax|)\psi^c(\xi+\xi_{d+1}Ax) 
\end{align*}
where  we remove the region $|\xi+\xi_{d+1}Ax|\le 1$  since its average exhibits smoother behavior.
Apply the coordinate change  $\xi \rightarrow  \xi_{d+1}\xi$ and Plancherel theorem for $\xi_{d+1}=\lambda$. Next
  use the  asymptotic expansion (\ref{kc3}) with the  error term  ignored. Then we replace $  [\partial^{\alpha}\mathcal{A}]$ with the sum over $j$ and $\nu=\pm $ of the following dyadic pieces:
\begin{align*}
\mathcal{A}_{\alpha,j}^{\lambda,t,\nu}f(x):=\frac{ \lambda^{d/2}\lambda^{\alpha}\chi\left(t\right)   }{2^{(d-1)j/2}}  \int_{\mathbb{R}^d}a_\nu(1+|\xi|)^{  \alpha}
e^{2\pi i\lambda \phi^\nu(x,t,\xi)} 
\chi\left(\frac{\lambda t|\xi+x|}{2^j}\right)    \widehat{f}(  \xi)d\xi\end{align*} 
where $a_+$ and $a_-$ are the initial constants of (\ref{kc3}), and $\Phi^\nu$ are phases:
\begin{align*}
\phi^+(x,t,\xi)&=\langle Ax,\xi\rangle+ t|\xi+x|\ \text{and}\ 
\phi^-(x,t,\xi)=\langle Ax,\xi\rangle-t|\xi+x|.
\end{align*}
Denote $\mathcal{A}^{\lambda,\nu }_{\alpha,j}f(x,t)=\mathcal{A}^{\lambda,t,\nu }_{\alpha,j}f(x)$ in the above.  Then we have
 \begin{align}\label{n32}
   \sup_{\|f\|_{L^2(\mathbb{R}^{d+1})}=1}\frac{\|[\partial^{\alpha}\mathcal{A}]f\|_{L^2(\mathbb{R}^{d+1}\times [1,2])}}{\left\|
f\right\|_{L^2(\mathbb{R}^{d+1})}}  
& \approx  \left\|\sum_{j=1}^\infty \mathcal{A}_{\alpha,j}^{\lambda,+}  +\mathcal{A}_{\alpha,j}^{\lambda,-}  \right\|_{L^2(\mathbb{R}^{d})\rightarrow L^2(\mathbb{R}^d\times [1,2]) } \nonumber \\
& \le\sum_{\nu=\pm}\sum_{j=1}^\infty \left\| \mathcal{A}_{\alpha,j}^{\lambda,\nu}    \right\|_{L^2(\mathbb{R}^{d})\rightarrow L^2(\mathbb{R}^d\times [1,2]) } .
\end{align}

Assume that (i) $\Lambda(A)\cap \mathbb{R}=\emptyset$ ($m_A=1$) or (ii)   $\Lambda(A)\cap \mathbb{R}\ne \emptyset$ with $n(\mu)=1$ for every $\mu\in \Lambda(A)\cap \mathbb{R}$ with $m_A$ in (\ref{m11}). To prove (\ref{0045}), we shall show that
\begin{align}\label{0421}
\left\| \mathcal{A}_{\alpha,j}^{\lambda,+}    \right\|_{L^2(\mathbb{R}^{d})\rightarrow L^2(\mathbb{R}^d\times [1,2]) } \le C2^{-\epsilon j}\ \text{if $\alpha<\frac{d-1}{2}-\frac{m_A-1}{2}$}. 
\end{align}
since $\nu=-$ is similar.  The restriction $t\in [1,2]$ enables us to   localize $|x|\le 10$.  
By the support condition $|\xi+x|\le  2^j/\lambda $, we have $$|\xi|\lesssim 1\ \text{ if $\lambda\gtrsim 2^j$ and  $|\xi|\lesssim 2^j/\lambda$ if $2^j\gg \lambda$}$$
in the symbol of $\mathcal{A}_{\alpha,j}^{\lambda,t,+}$ above, which controls the derivatives part:  
\begin{align} \label{cem}
  (\lambda(1+|\xi|) )^{\alpha} \lesssim \begin{cases}  \lambda^{ \alpha}\ \text{if $\lambda\gtrsim 2^j$}\\
2^{j \alpha}\ \text{if $\lambda\ll 2^j$}.
\end{cases}
.\end{align}
Thus one can observe that
\begin{align}\label{k45}
\mathcal{A}_{\alpha,j}^{\lambda,+} \approx \frac{(\max\{|\lambda|,2^j\} )^{\alpha}}{2^{(d-1)j/2} }\mathcal{T}_j^{\lambda}\ \text{for $\mathcal{T}_j^{\lambda}$ defined in (\ref{5ppa})}.
\end{align}
This with the three bounds  of $\mathcal{T}_j^{\lambda}$: 
$$ \text{1 in (\ref{546}),    $2^{(m_A-1)j/2} $ in  (\ref{51}) and $  \left(\frac{2^j}{\lambda}\right)^{d/2}$ in (\ref{545}),}$$ respectively  gives  the following three bounds
 \begin{align*} 
  \left\|  \mathcal{A}_{\alpha,j}^{\lambda,+}\right\|_{L^2(\mathbb{R}^{d})\rightarrow L^2(\mathbb{R}^d\times [1,2])}&\lesssim \begin{cases} 
2^{-(d-1)j/2} 2^{j  \alpha} \ \text{if $\lambda\ll  2^j$}\\
2^{-(d-1)j/2} 2^{ j \alpha}    2^{(m_A -1)j/2}\ \text{if $\lambda\approx 2^j$}\\
2^{-(d-1)j/2} \lambda^{  \alpha}    \left(\frac{2^j}{\lambda}\right)^{d/2}\le 2^{-(d-1)j/2} 2^{j\alpha}\ \text{if $\lambda\gg 2^j$}
\end{cases}
\end{align*}
all of which are  bounded by $ 2^{-\epsilon j}$ if
$
 \alpha= (d-1)/2-  (m_A-1)/2- \epsilon$ for some $\epsilon 
> 0$ 
showing  (\ref{0421}).   The case $\nu=-$ is similarly obtained.

 \subsection{$L^2(dx )$ estimate of skew symmetric $A$ for  fixed $t$}
By the Plancherel theorem with respect to $\xi_{d+1}$,  we have for each fixed $t>0$,
\begin{align*} 
\sup_{f\in L^2(\mathbb{R}^{d+1})}\frac{\|\mathcal{A}_{S^{d-1}(A)}f(\cdot,t)\|_{L^2_{\alpha}(\mathbb{R}^{d+1})}}{\left\|
f\right\|_{L^2(\mathbb{R}^{d+1})}}&\approx \sup_{\lambda }  \left\|\sum_{j=1}^\infty \mathcal{A}_{\alpha,j}^{\lambda,t,+}  +\mathcal{A}_{\alpha,j}^{\lambda,t,-}  \right\|_{L^2(\mathbb{R}^{d})\rightarrow L^2(\mathbb{R}^d) } \nonumber\\
& \le\sup_{\lambda } \sum_{\nu=\pm}\sum_{j=1}^\infty \left\| \mathcal{A}_{\alpha,j}^{\lambda,t,\nu}    \right\|_{L^2(\mathbb{R}^{d})\rightarrow L^2(\mathbb{R}^d) } .
\end{align*}
We shall first assume that $A$ is skew symmetric and prove the sufficient part below (\ref{a245}).  For this,  we shall prove that for $\alpha=\frac{d-1}{2}-\frac{1}{6}-\epsilon$, 
\begin{align}\label{1k3}
  \|   \mathcal{A}_{\alpha,j}^{\lambda,t,+}   \|_{L^2(\mathbb{R}^d)\rightarrow L^2(\mathbb{R}^d)}&\le C 2^{-\epsilon j}.
\end{align}
Under the condition of $\lambda\lesssim 2^j$, we have
\begin{align} \label{1k1}
  \left\|   \mathcal{A}_{\alpha,j}^{\lambda,t,+}  \right\|_{L^2(\mathbb{R}^{d})\rightarrow L^2(\mathbb{R}^d)}&\lesssim 
2^{j  \alpha} 2^{-(d-1)j/2}\le 2^{-j/6}.
\end{align}
\begin{proof}[Proof of (\ref{1k1})]
The case $\lambda\ll 2^j$ follows from (\ref{cem}),(\ref{k45}) and the second part of (\ref{546}).  The case $\lambda\approx 2^j$    from the H\"{o}rmander theorem with (\ref{449}) and (\ref{cem}).
\end{proof}
For the case $\lambda\gg 2^j$,  we claim that  there is $C$ independent of $\lambda$ such that
\begin{align} \label{1k2}
 \left\|   \mathcal{A}_{\alpha,j}^{\lambda,t,+}  \right\|_{L^2(\mathbb{R}^{d})\rightarrow L^2(\mathbb{R}^d)}&\le C 2^{-\epsilon j}\ \text{for $\alpha=\frac{d-1}{2}-\frac{1}{6}-\epsilon$}.
\end{align}
\begin{proof} [Proof of (\ref{1k2})]
By using the same dilation in (\ref{445}) and (\ref{k45}),   we work with
\begin{align*}
\mathcal{A}_{\alpha,j}^{\lambda,t,+} g( x ) &  =\left(\frac{2^j}{\lambda}\right)^{d}  \frac{ \lambda^{d/2}\lambda^{\alpha} }{2^{(d-1)j/2}} \chi\left(t\right)  \int_{\mathbb{R}^d}
e^{2\pi i \frac{2^{2j}}{\lambda}\left(  \langle Ax, \xi \rangle+t
\frac{\lambda}{2^j}|x+\xi|\right)  }
 \chi\left(  t|\xi+x|\right) \widehat{f
 }(  \xi)d\xi.
\end{align*} 
By writing $\alpha=\frac{d-1}{2}-a$ with $a=1/6+\epsilon$  in front of the above integral,
$$\left(\frac{2^j}{\lambda}\right)^{d}  \frac{ \lambda^{d/2}\lambda^{\alpha} }{2^{(d-1)j/2}}=\frac{2^{(d+1)j/2}}{\lambda^{1/2+a}}.$$
By applying the  H\"{o}rmander theorem for the phases in the two ways:
\begin{itemize}
\item $\det\left(\text{the mixed hessian of $\langle Ax, \xi \rangle+t
\frac{\lambda}{2^j}|x+\xi|$}\right)  \gtrsim  1+|\frac{\lambda}{2^j}|^{d-2}$ in (\ref{449}) 
\item   rank $d-1$ of the mixed hessian of $\frac{2^j}{\lambda}\langle Ax, \xi \rangle+t|x+\xi|$      with $\frac{2^j}{\lambda}\ll 1$,
\end{itemize}
we obtain that
\begin{align*}
 \| \mathcal{A}_{\alpha,j}^{\lambda,t,+} \|_{L^2(\mathbb{R}^d)\rightarrow L^2(\mathbb{R}^d)}\le &\frac{2^{(d+1)j/2}}{\lambda^{1/2+a}}   \min\left\{ \left(\left(\frac{2^{2j}}{\lambda}\right)  \left(\frac{\lambda}{2^{j}}\right)^{\frac{d-2}{d}}\right)^{-d/2}, 2^{-\frac{(d-1)j}{2}}\right\}  \\
&=\frac{2^{(d+1)j/2}}{\lambda^{1/2+a}} \min\left\{  2^{- (d+2)j/2 }\lambda, 2^{-\frac{(d-1)j}{2}}\right\} \le 2^{-3\epsilon j /2}.
\end{align*}
So, we obtain (\ref{1k2}).
\end{proof}
\subsection{Proof of $(\ref{a245})$ for $\Lambda(A)\cap \mathbb{R}= \emptyset$}
Let $ \mathcal{A}_{\alpha}^{\lambda,t,+} +\mathcal{A}_{\alpha}^{\lambda,t,-}=\sum_{j=1}^\infty \mathcal{A}_{\alpha,j}^{\lambda,t,+}  +\mathcal{A}_{\alpha,j}^{\lambda,t,-}  $. 
It suffices to fix $t=1$ and  prove that 
  \begin{align} \label{k567}
  \sup_{\lambda }  \left\| \mathcal{A}_{\alpha}^{\lambda,1,+} +\mathcal{A}_{\alpha}^{\lambda,1,-} \right\|_{L^2(\mathbb{R}^{d})\rightarrow L^2(\mathbb{R}^d) } =\infty\ \text{for $\alpha=\frac{d-1}{2}-1/6+\epsilon$}.
\end{align}
\begin{proof}[Proof of (\ref{k567})]
Take $\lambda\gg 1$. We shall choose $f$ and $g$ for computing
\begin{align*}
\langle \mathcal{A}_{\alpha}^{\lambda,1,+} f, g\rangle &  =   \lambda^{d/2}\lambda^{\alpha}   \int_{\mathbb{R}^d}
e^{2\pi i \lambda\left(  \langle Ax, \xi \rangle+ 
 |x+\xi|\right)  }
 \frac{1}{|\lambda (\xi+x)|^{(d-1)/2}} \widehat{f
 }(  \xi)g(x)d\xi dx
\end{align*} 
where $\lambda^{d/2}\lambda^{\alpha}=\lambda^{d-2/3+\epsilon} $. Fix $v\in S^{d-1}$ and $0<\rho\ll 1$.
Set   two functions
\begin{align*}
\widehat{f}( \xi)&:=(\lambda^{d/3})^{1/2}e^{-2\pi i \lambda   \langle \xi, v\rangle } \psi\left( \lambda^{2/3} |\xi-P_{v}(\xi) | /\rho\right)\prod_{i}\chi(\lambda^{1/3}\xi_i)\\
g( x)&:=(\lambda^{d/3})^{1/2}e^{-2\pi i \lambda \langle x, v\rangle } \psi\left(\lambda^{2/3} |\xi-P_{v}(\xi) |/\rho \right)\prod_{i}\chi(\lambda^{1/3}x_i).
\end{align*} 
Notice they have the same $L^2(\mathbb{R}^d)$ norms of
\begin{align*}
(\lambda^{-d/3})^{1/2}\widehat{f}(\lambda^{-1/3}\xi)&:=e^{-2\pi i \lambda^{2/3} \langle \xi, v\rangle } \psi\left(\frac{ |\xi-P_{v}(\xi) |}{\lambda^{-1/3}\rho}\right)\prod_{i}\chi( \xi_i)\\
(\lambda^{-d/3})^{1/2}g(\lambda^{-1/3}x)&:=e^{-2\pi i \lambda^{2/3} \langle x, v\rangle }\psi\left(\frac{ |x-P_{v}(x) |}{\lambda^{-1/3}\rho}\right)\prod_{i}\chi(x_i).
\end{align*}
Then by change of variable $\xi,x\rightarrow \lambda^{-1/3}\xi, \lambda^{-1/3}x$,
\begin{align*}
&\langle \mathcal{A}_{\alpha}^{\lambda,1,+} f, g\rangle   =  \frac{\lambda^{d-2/3+\epsilon}  \lambda^{-d/3} }{\lambda^{(d-1)/3}}  \int_{\mathbb{R}^d}
e^{2\pi i \left( \lambda^{1/3} \langle Ax, \xi \rangle+ 
  \lambda^{2/3}|x+\xi| \right)  } \frac{1}{|  (\xi+x)|^{(d-1)/2}}
  \\
&\times e^{-2\pi i  \lambda^{2/3}  \langle \xi, v\rangle  } \psi\left(\frac{ |\xi-P_{v}(\xi) |}{\lambda^{-1/3}\rho}\right)\prod_{i}\chi( \xi_i) 
e^{-2\pi i \lambda^{2/3} \langle x, v\rangle  }\psi\left(\frac{ |x-P_{v}(x) |}{\lambda^{-1/3}\rho}\right)\prod_{i}\chi( x_i)d\xi dx.
\end{align*}
where $|x+\xi|\approx 1$. Then we can rewrite it as  
\begin{align*}
\langle \mathcal{A}_{\alpha}^{\lambda,1,+} f, g\rangle  & \approx   \lambda^{(d-1)/3}   \lambda^{\epsilon}     \int_{\mathbb{R}^d}
e^{2\pi i \left( \lambda^{1/3} \langle Ax, \xi \rangle+ 
 \lambda^{2/3}(|x+\xi|-\langle (x+\xi), v\rangle  )\right)  }
 \\
&\times  \psi\left(\frac{ |\xi-P_{v}(\xi) |}{\lambda^{-1/3}\rho}\right)\prod_{i}\chi( \xi_i)
 \psi\left(\frac{ |x-P_{v}(x) |}{\lambda^{-1/3}\rho}\right)\prod_{i}\chi( x_i)d\xi dx.
\end{align*}
In the above,    $|\frac{ |x-P_{v}(x) |}{\lambda^{-1/3}\rho}|\le 1$ if and only if $x\in \text{cone}(v)+O(\rho \lambda^{-1/3})$.  Thus, $\langle Ax,\xi\rangle=\langle A(x-P_{v}(x)),\xi\rangle+ \langle A(P_{v}(x)),\xi-
P_{v}(\xi)\rangle=O(\lambda^{-1/3}\rho)$ because the skew symmetry of $A$ implies $ \langle A(P_{v}(x)), 
P_{v}(\xi)\rangle=\langle A(v),v\rangle =0$. This  combined with the good approximation $|x+\xi|-\langle (x+\xi),v\rangle=O(\eta^2)$   in (\ref{0442}) for $\eta=\lambda^{-1/3}\rho$, eliminates any oscillation in the above integral.
Therefore, one can obtain that
$$ \frac{|\langle \mathcal{A}_{\alpha}^{\lambda,1,+} f, g\rangle|}{\|f\|_{L^2(\mathbb{R}^d)}  \|g\|_{L^2(\mathbb{R}^d)} } \approx\frac{\lambda^{(d-1)/3} \lambda^{\epsilon} \lambda^{-2(d-1)/3} }{\lambda^{-(d-1)/3}\rho^{d-1} }  =\frac{\lambda^{\epsilon}}{\rho^{d-1}}\rightarrow \infty$$
as $\lambda\rightarrow \infty$. On the other hand,
\begin{align*}
\langle \mathcal{A}_{\alpha}^{\lambda,1, -} f, &g\rangle   =   \lambda^{(d-1)/3} \lambda^{\epsilon}   \int_{\mathbb{R}^d}
e^{2\pi i \left(\lambda^{1/3} \langle Ax, \xi \rangle- 
  \lambda^{2/3}(|x+\xi|+\langle (x+\xi), v\rangle)\right)  }
 \\
&\times  \psi\left(\frac{ |\xi-P_{v}(\xi) |}{\lambda^{-1/3}\rho}\right)\prod_{i}\chi( \xi_i)
 \psi\left(\frac{ |x-P_{v}(x) |}{\lambda^{-1/3}\rho}\right)\prod_{i}\chi( x_i)d\xi dx=O(\lambda^{-2N/3})
\end{align*}
since the derivatives $v\cdot\nabla_x $ and $v\cdot\nabla_\xi $  of $ |x+\xi|+\langle (x+\xi), v\rangle v$ are non-vanishing. Therefore by taking $\lambda$ large enough in the above, we obtain (\ref{k567}).
\end{proof}
\subsection{Proof of (\ref{a245}) for $\Lambda(A)\cap \mathbb{R}\ne \emptyset$}
Suppose that the maximal algebraic multiplicity of $A$ is $m_A$. By (\ref{c44}), it suffices to assume that $\text{rank}(\phi''_{x\xi}(x_0,t,\xi_0)= d-m_A$.
For $ \alpha_0:=\frac{d-1}{2}-\frac{m_A-1}{2}- \frac{1}{6}$, we claim that there exists $t\ne 0$   such that
\begin{align*} 
 \sup_{\lambda\ne 0}  \left\| \mathcal{A}_{\alpha}^{\lambda,t,+}  +\mathcal{A}_{\alpha}^{\lambda,t,-}  \right\|_{L^2(\mathbb{R}^{d})\rightarrow L^2(\mathbb{R}^d) }=\infty\ \text{if $\alpha>\alpha_0$. }
\end{align*}
Take $\lambda \gg 1$ and fix $t\ne 0$. We shall show $\|\mathcal{A}_{\alpha}^{\lambda,t,+}  +\mathcal{A}_{\alpha}^{\lambda,t,-}  \|_{op}\gtrsim\lambda^{-\alpha_0+\alpha}$.  Suppose that $\mu\in \Lambda(A)\cap \mathbb{R}$ and  $m_A=\text{dim}(\mathcal{E}(\mu))$.
Then by the proof of (\ref{c44}), there are  $x_0,\xi_0\approx 1$ with $|x_0+\xi_0|\approx 1$ such that $\frac{-t_0}{|x_0+\xi_0|}=\mu$ and $\text{rank}(\phi''_{x\xi}(x_0,t,\xi_0))\le d-m_A$.  We may assume that $\text{rank}(\phi''_{x\xi}(x_0,t,\xi_0)= d-m_A$.
Thus,
one can apply the coordinate change $(x,\xi)\rightarrow (Qx,Q\xi)$ so that $(Q[x]_0,Q[\xi]_0)=(x_0,\xi_0)$ and $[\phi(Qx,t,Q\xi)]''_{x\xi}([x]_0,[\xi]_0)=Q^T
[\phi]''_{x\xi}(x_0,t,y_0)Q$ is a diagonal matrix whose diagonal vector is $(\beta_1,\cdots,\beta_{d-m_A},0,\cdots,0)$. For simplicity, we regard $Q=I$ and $\phi^+(x_0,t,\xi_0)=0$.  By  Taylor expansion near $(x_0,\xi_0)$ for $x_0=(x_{s0})_s,\xi_0=(\xi_{s0})_s$,  
\begin{align*}
\phi^+(x,t,\xi)&=( A^T\xi_0+t\frac{x_0+\xi_0}{|x_0+\xi_0|})\cdot (x-x_0) +(Ax_0+t\frac{x_0+\xi_0}{|x_0+\xi_0|})\cdot (\xi-\xi_0) \\
&+\sum_{s=1}^{d-m_A} \beta_s(x_s-x_{s0})(\xi_s-\xi_{s0})
+O(|x-x_0|^2|\xi-\xi_0| +|\xi-\xi_0|^2|x-x_0|).
\end{align*}
Take a cutoff function $\psi(u)$ supported in $|u|\ll 1$ and choose $g(x):=g_{\text{phase}}(x)g_{\text{supp}}(x)$ and $\widehat{f}(\xi):=\widehat{f}_{\text{phase}}(\xi) \widehat{f}_{\text{supp}}(\xi)$:
\begin{align*}
g_{\text{phase}}(x)&:=e^{-2\pi i   ( A^T\xi_0+t_0\frac{x_0+\xi_0}{|x_0+\xi_0|})\cdot (x-x_0)}\\
g_{\text{supp}}(x)&:=\prod_{s=1}^{d-m_A}\psi((x_s-x_{s0})\lambda^{1/2})\prod_{s=d-m_A+1}^{d}\psi((x_s-x_{s0})\lambda^{1/3})\\
\widehat{f}_{\text{phase}}(\xi)
 &:=e^{-2\pi i   ( Ax_0+t_0\frac{x_0+\xi_0}{|x_0+\xi_0|})\cdot (\xi-\xi_0)}\\
\widehat{f}_{\text{supp}}(\xi)&:=
\prod_{s=1}^{d-m_A}\psi((\xi_s-\xi_{s0})\lambda^{1/2})\prod_{s=d-m_A+1}^{d}\psi((\xi_s-\xi_{s0})\lambda^{1/3}).
\end{align*}
Then $\int g_{\text{supp}}(x) dx=   \int \widehat{f}_{\text{supp}}(\xi) d\xi\approx \lambda^{-\frac{1}{2}(d-m_A)} \lambda^{-\frac{1}{3}m_A}$.
Observe that the linear terms in the phase $\phi^+$ are absorbed into the phase parts of $g$ and $\widehat{f}$ of  the integral $|\langle \mathcal{A}_{\alpha}^{\lambda,t,+}f,g\rangle|$ which is expressed as
\begin{align*}
& \frac{\lambda^{d/2}\lambda^{ \alpha}}{\lambda^{(d-1)/2}}\left|\int   \frac{  e^{  2\pi i \lambda( \sum_{s=1}^{d-m_A} \beta_s(x_s-x_{s0})(\xi_s-\xi_{s0})
+\text{error term})  }     }{  (|\xi|+ 1)^{-\alpha} |\xi+x|^{(d-1)/2}}g_{\text{supp}}(x)
\widehat{f}_{\text{phase}}(\xi)
dx d\xi \right|\\
&\approx \frac{\lambda^{d/2}\lambda^{ \alpha}}{\lambda^{(d-1)/2}}\int g_{\text{supp}}(x) dx\times \int \widehat{f}_{\text{supp}}(\xi) d\xi.
\end{align*}
However, the linear terms in the phase $\Phi^-$ remains. This enables us to  apply integration by parts and obtain that  
$|\langle \mathcal{A}_{\alpha}^{\lambda,t,-} f,g\rangle|$ is
\begin{align*}
& \frac{\lambda^{d/2}\lambda^{ \alpha}}{\lambda^{(d-1)/2}}\big|\int   
\frac{e^{2\pi i \lambda \left(2t \frac{(x_0+\xi_0)}{|x_0+\xi_0|}
\cdot [(x-x_0)+(\xi-\xi_0)]+ \sum_{s=1}^{d-m_A} \beta_s(x_s-x_{s0})(\xi_s-\xi_{s0})
+\text{error term})    \right)}}{(|\xi|+ 1)^{-\alpha} |\xi+x|^{(d-1)/2}}  \\
& \times g_{\text{supp}}(x)
\widehat{f}_{\text{phase}}(\xi)
dx d\xi  =O(\lambda^{-N})
\end{align*}
where $ \frac{(x_0+\xi_0)}{|x_0+\xi_0|} \cdot \nabla_x (\text{the above phase})\approx 1$ and $ \frac{(x_0+\xi_0)}{|x_0+\xi_0|} \cdot \nabla_\xi(\text{the above phase})\approx 1$ due to the linear terms. Thus  $|\langle \mathcal{A}_{\alpha}^{\lambda,t,+}f,g\rangle|$ is the main term showing for $\alpha>\alpha_0$,
\begin{align*}
\frac{|\langle (\mathcal{A}_{\alpha}^{\lambda,t,+}+\mathcal{A}_{\alpha}^{\lambda,t,-})f,g\rangle|}{\|f\|\|g\|}&\approx \frac{\lambda^{d/2}\lambda^{ \alpha}}{\lambda^{(d-1)/2}}\lambda^{-\frac{1}{2}(d-m_A)} \lambda^{-\frac{1}{3}} \approx \lambda^{\frac{1}{2}-\frac{1}{2}(d-m_A)-\frac{1}{3} +\alpha}=\lambda^{\alpha-\alpha_0}
\end{align*}
 which goes to $\infty$ as $\lambda\rightarrow \infty$. Therefore we proved (\ref{a245}). 
\section{Appendix}\label{sec10}
\subsection{Proof of (\ref{kcc2})}
By applying the standard majorization argument, it suffices to obtain  an upper bound $C_p$ of the  norm  in  $L^p(\mathbb{R}^{d+1})\rightarrow L^p(\mathbb{R}^{d+1})$ for the maximal operator $f\rightarrow \sup_{k\in\mathbb{Z}}| \mathcal{W}_k(f)|$ for every $p>1$, where
\begin{align}
 \mathcal{W}_k(f)(x,x_{d+1})=2^{kd}\int_{|y|\le 2^{-k}} f(x-y,x_{d+1}-\langle Ax, y\rangle)dy.  \label{2.3q}
\end{align}
To restrict the frequency of the last variable,  we take $$\Phi_k(y,y_{d+1})=\delta(y)2^{2k}\psi^{\vee}(2^{2k}y_{d+1})\  \text{and}\ \Psi_{k+\ell}(y,y_{d+1})=\delta(y)2^{2(k+\ell)}\chi^{\vee}(2^{2(k+\ell)}y_{d+1}). $$
Then from the identity
$1\equiv \psi(\xi_{d+1}/2^{2k})+\sum_{\ell=1}^{\infty} \chi(\xi_{d+1}/2^{2(k+\ell)})$,
 one has
\begin{align}\label{00kk}
\sup_{k\in\mathbb{Z}} | \mathcal{W}_kf | 
\le    \sup_{k\in\mathbb{Z}}  |\mathcal{W}_k(\Phi_k\circ_A f)|+ \sum_{\ell=1}^\infty \left( \sum_{k\in\mathbb{Z}}|\mathcal{W}_k( \Psi_{k+\ell} \circ_A f)|^2\right)^{1/2}.
\end{align}
Using the Vitali covering for  rectangles of sizes $2^{-k}\times\cdots 2^{-k}\times 2^{-2k}$ in $\mathbb{G}_A \cong \mathbb{R}^{d+1}$,
\begin{align*} 
\|\sup_k   |\mathcal{W}_k(\Phi_k\circ_A f)| \,\|_{L^{1,\infty}(\mathbb{R}^{d+1})}\le C\|f\|_{L^1(\mathbb{R}^{d+1})}.
\end{align*}
This with its $L^{\infty}$-operator norm $\approx 1$ enables us to apply the Marcinkiewicz interpolation theorem to obtain the desired $L^p$ bound of the  maximal operator 
\begin{align}\label{000kk}
f\rightarrow \sup_k   |\mathcal{W}_k(\Phi_k\circ_A f)|\ \text{in (\ref{00kk}).}
\end{align}
Next,   for $\ell\ge 1$, one can obtain the square function estimates:
\begin{align}\label{2bd1}
\|\left(\sum_{k}|\mathcal{W}_k(\Psi_{k+\ell}\circ_A f)|^2\right)^{1/2}\|_{L^{2}(\mathbb{R}^{d+1})}&\le C2^{-d\ell }\|f\|_{L^2(\mathbb{R}^{d+1})}\ \end{align}
from  the disjointness of the supports  of $ \chi(\xi_{d+1}/2^{2(k+\ell)})$   over all $k$'s, combined with the bound $\|T^{\lambda}_{\ell} \|_{op}=O(2^{-d\ell})$ of the operator $g\rightarrow T^{\lambda}_\ell(g)$ with $\lambda=\xi_{d+1}$, defined by
$$T^{\lambda}_\ell g(x):= \lambda^{d/2}\left(\frac{2^{k}}{\lambda}\right)^d\int e^{2\pi i \lambda (\frac{\lambda}{2^k})^2 \langle A^{-1}x,\xi\rangle}\chi(x+\xi)\widehat{g}(\xi)d\xi\  $$
where 
the norm $2^{-d\ell}$ follows directly from  Schur's test and $\lambda\sim 2^{2k+2\ell}$. Thus one can combine (\ref{2bd1}) with the $L^2$ boundedness of the maximal operator  in (\ref{000kk})  to show the $L^2(\mathbb{R}^{d+1})$ boundedness of $f\rightarrow \sup_k|\mathcal{W}_k(f)|$ in (\ref{00kk}).
Moreover, for $4/3<p\le 2$, we apply 
\begin{itemize}
\item the vector valued inequality  in $L^p(\ell^2)$  for $(f_k)_k\rightarrow (\mathcal{W}_k(f_k))_k$ due to the $L^2(\mathbb{R}^{d+1})$ boundedness of $f\rightarrow \sup_k|\mathcal{W}_k(f)| $ and  
\item the Littlewood-Paley inequality for $f\rightarrow \Psi_{k+\ell}\circ_A f$ 
\end{itemize}
to have
\begin{align}\label{pbd2}
\|\left(\sum_{k}|\mathcal{W}_k(\Psi_{k+\ell}\circ_A f)|^2\right)^{1/2}\|_{L^{p}(\mathbb{R}^{d+1})}
&\le C \|f\|_{L^p(\mathbb{R}^{d+1})}. \end{align}
 This with (\ref{2bd1}) leads the $L^p(\mathbb{R}^{d+1})$ boundedness of $f\rightarrow \sup_k|\mathcal{W}_k(f)|$ for  $4/3<p\le 2$.    Repeat the argument to obtain $L^p$ boundedness of the maximal operator for the full range $1<p\le 2$.\subsection{Proof of Lemma \ref{prop23} }\label{app11.2}
  For   $(j,k)\in\mathbb{Z}_+\times\mathbb{Z}$,   we claim (\ref{yb1})  with $ 
  \tilde{\mathcal{A}}_{j,k}$ 
  $$ \tilde{\mathcal{A}}_{j,k}f(x,x_{d+1},t)=\int e^{2\pi i(x,x_{d+1})\cdot (\xi,\xi_{d+1})}m_{j,k}(x,x_{d+1},t,\xi,\xi_{d+1})\widehat{f}(\xi,\xi_{d+1})d\xi d\xi_{d+1}$$ with $m_{j,k}(x,x_{d+1},t,\xi,\xi_{d+1})$ given by $\chi\left(\frac{t|\xi+\xi_{{d+1}}Ax|}{2^j} \right)\chi(2^kt) $ times like  $\widehat{d\tilde{\sigma}}(2\pi t(\xi+\xi_{d+1}Ax)  $ as in (\ref{i7}).  
Given $p\ge 1$, we set the dilations as
\begin{align}
[\mathcal{D}^{p}_{2^{-k}}g](x,x_{d+1},t)&=2^{-(d+3)k/p}g(2^{-k}x,2^{-2k}x_{d+1},2^{-k}t),\label{4bv}\\
[D^{p}_{2^{-k}}f](x,x_{d+1})&=2^{-(d+2)k/p}f(2^{-k}x,2^{-2k}x_{d+1})\nonumber
\end{align}
satisfying the following two  $L^p$-invariance conditions:
\begin{align}\label{nk}
\|g\|_{L^p(\mathbb{R}^{{d+1}}\times \mathbb{R})}=
\|\mathcal{D}^{p}_{2^{-k}}g\|_{L^p(\mathbb{R}^{{d+1}}\times
 \mathbb{R})}\ \text{and}\ \|f\|_{L^p(\mathbb{R}^{{d+1}})}=
 \|D_{2^{-k}}f\|_{L^p(\mathbb{R}^{{d+1}})}.
 \end{align}
In (\ref{i7}),   we can check that
\begin{align}\label{i0}
m_{j,k}(2^{-k}x,2^{-2k}x_{d+1},2^{-k}t,2^{k}\xi,2^{2k}\xi_{d+1})= m_{j,0}( x, x_{d+1}, t,\xi,\xi_{d+1}).
\end{align}
Then     it holds that
\begin{align}
&\mathcal{D}^{p}_{2^{-k}}[2^{k/p} \tilde{\mathcal{A}}_{j,k} f](x, x_{d+1},t)=   \tilde{\mathcal{A}}_{j,0} D^p_{2^{-k}} f(x,x_{d+1},t).
\label{nka}\\
&\mathcal{D}^{p}_{2^{-k}}[2^{k/p} \tilde{\mathcal{A}}_{j,k} \mathcal{P}_{j+k+\ell}f](x, x_{d+1},t)  
= \tilde{\mathcal{A}}_{j,0}\mathcal{P}_{j+\ell}D^p_{2^{-k}} f(x, x_{d+1},t).\label{nka11}
\end{align}
\begin{proof}[Proof of (\ref{nka})]
We show (\ref{nka}) for the symbols $m_{j,k}$ satisfying   (\ref{i0}). 
By using $ \mathcal{D}^{p}_{2^{-k}}$ in (\ref{4bv}), we write
\begin{align*}
 \mathcal{D}^{p}_{2^{-k}}[2^{k/p}\tilde{\mathcal{A}}_{j,k}  f](x, x_{d+1},t)&= 2^{-(d+2)k/p}      \int
e^{2\pi i\left((2^{-k}x,2^{-2k}x_{d+1})\cdot(\xi,\xi_{d+1}) \right)}\nonumber\\
&\times m_{j,k}(2^{-k}x,2^{-2k}x_{d+1},2^{-k}t, \xi, \xi_{d+1})
\widehat{f}( \xi, \xi_{d+1})d\xi d\xi_{d+1}.
 \end{align*}
This integral after the change of variable $ (\xi,\xi_{d+1}) \rightarrow (2^{k}\xi,2^{2k}\xi_{d+1}) $ becomes
\begin{align*}
&   2^{-(d+2)k/p}    2^{(d+2)k/p}  \int
e^{2\pi i\left((2^{-k}x,2^{-2k}x_{d+1})\cdot(2^k\xi,2^{2k}\xi_{d+1}) \right)}\nonumber\\
&\times m_{j,k}(2^{-k}x,2^{-2k}x_{d+1},2^{-k}t, 2^k\xi, 2^{2k}\xi_{d+1})
 2^{(d+2)k/p'} \widehat{f}(2^k\xi, 2^{2k}\xi_{d+1})d\xi d\xi_{d+1}.
 \end{align*}
Hence, by $   2^{(d+2)k/p'} \widehat{f}(2^k\xi, 2^{2k}\xi_{d+1}) =\left[D^p_{2^{-k}} f\right]^{\vee}
 (\xi,\xi_{d+1}) $ due to  (\ref{4bv}), and (\ref{i0}),
\begin{align*}
 & \mathcal{D}^{p}_{2^{-k}}[2^{k/p} \tilde{\mathcal{A}}_{j,k}  f](x, x_{d+1},t)\\
 &\qquad\quad= \int
 e^{2\pi i\left((x,x_{d+1})\cdot(\xi,\xi_{d+1}) \right)}
m_{j,0}(x,x_{d+1},t, \xi, \xi_{d+1}) \left[D^p_{2^{-k}} f\right]^{\vee}
 (\xi,\xi_{d+1})d\xi d\xi_{d+1}
 \end{align*}
which is $\tilde{\mathcal{A}}_{j,0} D^p_{2^{-k}} f(x,x_{d+1},t).$  Thus we proved (\ref{nka}).  \end{proof}
\begin{proof}[Proof of (\ref{nka11})]
Apply the dilation (\ref{4bv}) to   $ \mathcal{P}_{j+\ell+k}f$ and obtain  the identity
\begin{align*} 
D^p_{2^{-k}} \mathcal{P}_{j+\ell+k}f(y,y_{d+1})=\mathcal{P}_{j+\ell}D^p_{2^{-k}} f(y,y_{d+1}).
\end{align*}
This with $\mathcal{D}^{p}_{2^{-k}}[2^{k/p} \tilde{\mathcal{A}}_{j,k}]  =  \tilde{\mathcal{A}}_{j,0} D^p_{2^{-k}}$ in  (\ref{nka}) implies $\mathcal{D}^{p}_{2^{-k}}[2^{k/p} \tilde{\mathcal{A}}_{j,k} \mathcal{P}_{j+k+\ell}f] =  [\tilde{\mathcal{A}}_{j,0} D^p_{2^{-k}} \mathcal{P}_{j+k+\ell}f] 
= \tilde{\mathcal{A}}_{j,0} \mathcal{P}_{j+\ell}D^p_{2^{-k}} f$ in (\ref{nka11}).
\end{proof}
  From (\ref{nka}),(\ref{nka11}) with (\ref{nk}), we obtain (\ref{yb1}).

\subsection{Proof of Lemma \ref{prop1}}\label{Sec114}
Recall the Littlewood-Paley  projection $\mathcal{P}_jf$:
\begin{align*}
&\mathcal{P}_jf(y,y_{{d+1}})\\
&=\int_{(\eta,\eta_{d+1})\in \mathbb{R}^d\times \mathbb{R}}e^{2\pi i (\eta,\eta_{{d+1}})\cdot (y,y_{{d+1}}) }P_j(\eta+\eta_{d+1}Ay,\eta_{d+1})\widehat{f}(\eta,\eta_{{d+1}})d\eta d\eta_{{d+1}}
\end{align*}
where $P_j(\eta,\eta_{d+1})=\psi\left( \frac{\eta}{2^{j+1}},\frac{\eta_{d+1}}{2^{2(j+1)}}\right)- \psi\left( \frac{\eta}{2^{j}},\frac{\xi_{d+1}}{2^{2j}}\right)$  in (\ref{10061}). We only prove (\ref{00453}), since (\ref{045312}) is standard.
From (\ref{yb1}), it suffices to show 
\begin{align}
\left\|  \tilde{\mathcal{A}}_{j,0}\mathcal{P}_{j+\ell} \right\|_{ L^2(\mathbb{R}^{d+1})\rightarrow L^2(\mathbb{R}^{d+1}\times \mathbb{R})} &\lesssim 2^{-c|\ell|}  \ \text{for $|\ell|\ge 100d   j$,} \label{045b}
\end{align}
\noindent
We split $\tilde{\mathcal{A}}_{j,0}=\mathcal{A}^{1,\ell}_{j,0}+\mathcal{A}^{2,\ell}_{j,0}$ having  symbols $m_{j,0}=a^1_{j,\ell}+a^2_{j,\ell}$
  where
\begin{align*} 
a^1_{j,\ell} (x,x_{d+1},t,\xi,\xi_{d+1})&=  \psi\left(\frac{|\xi_{d+1}|}{ 2^{|\ell|/10} }\right)\widehat{d\tilde{\sigma}}(2\pi  t(\xi+\xi_{{d+1}}Ax)\chi(t)  \chi\left(\frac{t|\xi+\xi_{{d+1}}Ax|}{2^j} \right) ,\\
a^2_{j,\ell} (x,x_{d+1},t,\xi,\xi_{d+1})&=\psi^c\left(\frac{|\xi_{d+1}|}{ 2^{|\ell|/10} }\right)\widehat{d\tilde{\sigma}}(2\pi  t(\xi+\xi_{{d+1}}Ax)\chi(t)  \chi\left(\frac{t|\xi+\xi_{{d+1}}Ax|}{2^j} \right).\end{align*}
\noindent   {\bf Proof of (\ref{045b}) for $\mathcal{A}^{2,\ell}_{j,0}$}.  We show that
\begin{align} \label{1085}
\left\| \mathcal{A}^{1,\ell}_{j,0} f \right\|_{L^2(\mathbb{R}^{d+1}\times I)} \lesssim 2^{-c|\ell|}   \|f\|_{L^2(\mathbb{R}^{d+1})}\ \text{for all $|\ell|\ge 100 d j$.} 
\end{align} 
\begin{proof}[Proof of (\ref{1085})]
The last variable $\xi_{d+1}$ of the  symbol  $a^2_{j,\ell}$ is  $\lambda$ in $\mathcal{T}^\lambda_j$. By the support condition $|\lambda|\ge 2^{|\ell|/10}\ge 2^{ 10dj}$,  we   apply   (\ref{545}) to obtain that
\begin{align*} 
\left\| \mathcal{A}^{2,\ell}_{j,0}   \right\|_{L^2(\mathbb{R}^{d+1})\rightarrow L^2(\mathbb{R}^{d+1}\times I)} \le \frac{\sup_{|\lambda|\ge 2^{|\ell|/10}}\left\|\mathcal{T}^{\lambda}_j\right\|_{L^2(\mathbb{R}^d)\rightarrow L^2(\mathbb{R}^{d}\times\mathbb{R}) }}{2^{(d-1)j/2}} =    O\left(\frac{2^j}{\lambda}\right)   
\end{align*}
 which is  $O(2^{-c|\ell|})$ showing (\ref{1085}). 
 \end{proof}
\noindent
 {\bf Proof of (\ref{045b}) for $\mathcal{A}^{1,\ell}_{j,0} $}. 
It suffices to fix $t=1$ in (\ref{i7}) and prove 
\begin{align} \label{045bp}
\left\| \mathcal{A}^{1,\ell}_{j,0}\mathcal{P}_{j+\ell} f(\cdot,\cdot,1) \right\|_{L^2(\mathbb{R}^{d+1})} \lesssim 2^{-c|\ell|}   \|f\|_{L^2(\mathbb{R}^{d+1})}\ \text{for  $|\ell|\ge 100 d j$.} 
\end{align}
Our proof  is based on the argument of   M. Christ in \cite{Ch}, where he facilitated the nonisotropic dilations combined with the cancellation property of the singular kernels  in  the nilpotent Lie groups. 
Note   $\mathcal{A}^{1,\ell}_{j,0}\mathcal{P}_{j+\ell} f=f\cdot_{A}K_{j,\ell}\circ_AP_{j+\ell} $  where
\begin{align}
&\begin{cases}\mathcal{A}_{j,0}^{1,\ell}f(x,x_{d+1},1) 
 =f\circ_A K_{j,\ell}(x,x_{d+1})\\
\mathcal{P}_{j+\ell} f (x,x_{d+1})=f\circ_A P_{j+\ell}(x,x_{d+1})
\end{cases}\ \text{where} \qquad\qquad\nonumber\\
K_{j,\ell}(y,y_{d+1})&=\int e^{2\pi i  (\xi,\xi_{d+1})\cdot (y,y_{d+1})}
 \widehat{(d\tilde{\sigma}_{j})_1}(\xi) \psi\left(\frac{|\xi_{d+1}|}{ 2^{|\ell|/10} }\right)d\xi d\xi_{d+1}\label{92s}\\
 P_{j+\ell}(y,y_{d+1})&=2^{(j+\ell+1)d}2^{2(j+\ell+1)}  \psi^{\vee}(2^{j+\ell+1}y, 2^{2(j+\ell+1)}y_{d+1})   \label{10057}\\
 &-
 2^{(j+\ell)d}2^{2(j+\ell)}  \psi^{\vee}(2^{j+\ell}y, 2^{2(j+\ell)}y_{d+1}).
\nonumber
\end{align}
Recall  $f\circ_Ag(x,x_{d+1})=\int f(x-y,x_{d+1}-y_{d+1}-\langle Ax,y\rangle)g(y,y_{d+1})dydy_{d+1}$.
 We shall prove (\ref{045bp}) by Lemma \ref{lem599}-\ref{lem602} below.
 \begin{lemma}\label{lem599}
The  two kernels  in (\ref{92s}) and (\ref{10057}) satisfy the cancellation 
\begin{align}\label{10055}
\int K_{j,\ell}(y,y_{d+1})dydy_{d+1}=\int P_{j+\ell}(y,y_{d+1})dydy_{d+1}=0
\end{align}
and concentration properties:
   \begin{align}\label{10056}
 | K_{j,\ell}(y,y_{d+1})|\lesssim &\frac{2^j}{(1+|y|)^N}
  2^{|\ell|/10}\psi^{\vee}(2^{|\ell|/10}y_{d+1})\ \text{ having   $L^1$ norm $O(2^{j})$} \end{align}
  with the support of $P_{j+\ell}$ in (\ref{10057}) is essentially  supported in
  \begin{align}\label{10058}
  \text{ $\{(y,y_{d+1}): |(2^{j+\ell}y,2^{2(j+\ell)}y_{d+1})|\lesssim1\}$  and $\|P_{j+\ell}\|_{L^1}
=O(1)$.   }
  \end{align}
  \end{lemma}
    \begin{proof}
The first part  in (\ref{10055}) follows from $\widehat{K_j}(0) =0$ due to the integral $dy$ with the support condition $|\xi|\approx 2^j$, where  $\ \widehat{}\ $ indicates the Euclidean Fourier transform. The second part follows from
$\widehat{P_{j+\ell}}(0)=0$.
Note that $ \widehat{(d\tilde{\sigma}_{j})_1}(\xi) $   defined in (\ref{i7}) as
$$ \left[\int_{S^{d-1}}  e^{2\pi i t \xi \cdot z} \,  (\xi/|\xi|)\cdot z   d\sigma(z)+\widehat{d\sigma}(\xi)\right]\chi\left(\frac{\xi}{2^j}\right).
$$ This with (\ref{2.3q}) implies (\ref{10056}).
   Finally,   (\ref{10058}) follows from (\ref{10057}).
  \end{proof}
  \begin{lemma}\label{lem600}
Let   $\mathcal{A}_{j,0}$ and $\mathcal{P}_{j+\ell}$ be defined in  (\ref{92s}) and (\ref{10057}).  Then it holds
\begin{align}\label{kkgg}
\mathcal{A}_{j,0}\mathcal{P}_{j+\ell}f(x,x_{d+1},1)& =f\circ_A P_{j+\ell}\circ_A K_{j,\ell}(x,x_{d+1})=
f\circ_A(P_{j+\ell}*_{\mathbb{G}_{-A^T}}K_{j,\ell})(x,x_{d+1})\nonumber\\
&=\int f\left(x-y,x_{d+1}-y_{d+1}-\langle Ax,y\rangle \right)U_{j,\ell}(y,y_{d+1})dydy_{d+1}.
\end{align}
where
  $U_{j,\ell}(y,y_{d+1})=P_{j+\ell}*_{\mathbb{G}_{-A^T}}K_{j,\ell}(y,y_{d+1})$.
\end{lemma}
  \begin{proof}
It suffices to show the second inequality above, which follows from
\begin{align}\label{0cc}
(f\circ_A P)\circ_A K=f\circ_A( P*_{\mathbb{G}_{-A^T}}K).
\end{align}
To show  (\ref{0cc}), from   $f\circ_A g =\mathcal{U}_{-A}(\mathcal{U}_{A}f*_{\mathbb{G}_{-A^T}}g)$ in   (\ref{isig}), we obtain that
\begin{align*}
(f\circ_A P)\circ_A K&=\mathcal{U}_{-A}[\mathcal{U}_{A}(f\circ_A P)*_{\mathbb{G}_{-A^T}}K]\\
&=\mathcal{U}_{-A}[\mathcal{U}_{A}[ \mathcal{U}_{-A}(\mathcal{U}_{A}f*_{\mathbb{G}_{-A^T}} P)]  *_{\mathbb{G}_{-A^T}}K]\\
&=\mathcal{U}_{-A}[ (\mathcal{U}_{A}f*_{\mathbb{G}_{-A^T}} P)  *_{\mathbb{G}_{-A^T}}K]= \mathcal{U}_{-A}[ \mathcal{U}_{A}f*_{\mathbb{G}_{-A^T}} P  *_{\mathbb{G}_{-A^T}}K]\\
& = \mathcal{U}_{-A}[ \mathcal{U}_{A}f*_{\mathbb{G}_{-A^T}} (P  *_{\mathbb{G}_{-A^T}}K)]=f\circ_A (P  *_{\mathbb{G}_{-A^T}}K)
\end{align*}
We have the last line of  (\ref{kkgg})
from   $\circ_A$ and $*_{\mathbb{G}_{A}}$ in (\ref{isig}).
  \end{proof}
\begin{lemma}\label{lem602}
Let $|\ell| \ge 100d  j$.  Suppose that $U_{j,\ell}$ is defined in (\ref{kkgg}). Then   \begin{align}\label{100301}
\int |U_{j,\ell}(y,y_{d+1})|dy dy_{d+1}\lesssim 2^{-c|\ell|}\ \text{for some $c>0$.}
\end{align}
\end{lemma}
\begin{proof}
Note from (\ref{isig}), we  write $U_{j,\ell}(y,y_{d+1})=P_{j+\ell}*_{\mathbb{G}_{-A^T}}K_{j,\ell}(y,y_{d+1})$ as
\begin{align}\label{sspp}
& \int P_{j+\ell}\left(y-z,y_{d+1}-z_{d+1}+\langle  A^T (y-z),z\rangle \right)K_{j,\ell}(z,z_{d+1})dzdz_{d+1}\nonumber\\
 &= \int K_{j,\ell}\left(y-z,y_{d+1}-z_{d+1}+\langle  A (y-z),z\rangle \right)P_{j+\ell}(z,z_{d+1})dzdz_{d+1} 
\end{align}{\bf Case 1}. Let $\ell>0$. As $\int P_{j+\ell}(z,z_{d+1})dzdz_{d+1}=0$ in (\ref{sspp}),  write 
$U_{j,\ell}(y,y_{d+1})$:
\begin{align*}
  \int \bigg[K_{j,\ell}\left(y-z,y_{d+1}-z_{d+1}+\langle  A (y-z),z\rangle \right)-K_{j,\ell}\left(y,y_{d+1}  \right)\bigg] P_{j+\ell}(z,z_{d+1})dzdz_{d+1}.   
\end{align*}
Note that   $K_{j,\ell}$ is supported essentially on  $S^{d-1}\times [-2^{-|\ell|/10},2^{-|\ell|/10}]$ in (\ref{10056}). From this we observe that
\begin{itemize}
\item  $|\nabla_y K_{j,\ell}(y,y_{d+1})|\lesssim 2^j\times |K_{j,\ell}(y,y_{d+1})|$  because of (\ref{92s})  
\item $|\nabla_{y_{d+1}} K_{j,\ell}(y,y_{d+1})|\lesssim  2^{|\ell|/10}\times |K_{j,\ell}(y,y_{d+1})|$  because of (\ref{92s})     
\item $|z|, |\langle  A(y-z),z\rangle|\lesssim 2^{-j-\ell} $ and $ |z_{d+1}|   \le 2^{-2j-2|\ell|}$ due to (\ref{10057}) and (\ref{10056}). 
\end{itemize}
By utilizing this condition with $y$ dominating over $z$,   we apply the mean value theorem   for   the function $K_{j,\ell}$  to obtain that
\begin{align}
 \int |U_{j,\ell}(y,y_{d+1})|&dy dy_{d+1}\lesssim \int  |2^{-j-\ell}2^j+2^{-j-\ell} 2^{|\ell|/10}|\times |\tilde{K}_{j,\ell}(y,y_{d+1})| dydy_{d+1}\nonumber\\
&\times\int |P_{j+\ell}(z,z_{d+1})|dzdz_{d+1}=O(2^{j}2^{-8|\ell|/10})=O(2^{-|\ell|/2})\label{100302}
\end{align}
where $\tilde{K}_{j,\ell}$ is a slight change of $K_{j,\ell}$ in  (\ref{92s}).\\
{\bf Case 2}. Let $\ell<0$. By subtracting  the vanishing term below from $U_{j,\ell}(y,y_{d+1})$,  
$$ \int K_{j,\ell}\left(z,z_{d+1} \right)  P_{j+\ell}\left(y,y_{d+1}  \right) dzdz_{d+1} =0,$$
we rewrite $U_{j,\ell}(y,y_{d+1})$ in the first line of (\ref{sspp}) as
 $$ \int K_{j,\ell}\left(z,z_{d+1} \right) \bigg[ P_{j+\ell}\left(y-z,y_{d+1}-z_{d+1}+  \langle A^T(y-z),z\rangle\right)- P_{j+\ell}\left(y,y_{d+1}  \right) \bigg]  dzdz_{d+1}.$$
Note that for $\ell<0$ in (\ref{10057}) and (\ref{10056}),
\begin{itemize}
\item  $|\nabla_y P_{j+\ell}(y,y_{d+1})|\lesssim 2^{j-|\ell|}\times |P_{j+\ell}(y,y_{d+1})|$,
\item
  $|\nabla_{y_{d+1}} P_{j+\ell}(y,y_{d+1})|\lesssim 2^{2j-2|\ell|}\times |P_{j+\ell}(y,y_{d+1})|$, 
  \item $|z|\lesssim 1$,$|z_{d+1}|\lesssim 2^{- |\ell|/10}$,$|y-z|\lesssim 2^{-j+\ell}$ and $|\langle  A^T(y-z),z\rangle|\lesssim |y-z|\,|z|\lesssim 2^{-j+|\ell|}$.
\end{itemize}
We apply  the mean value theorem above to obtain
\begin{align*}
&\bigg[P_{j+\ell}\left(y-z,y_{d+1}-z_{d+1}+\langle  A(z),y-z
\rangle \right)-P_{j+\ell}\left(y,y_{d+1} \right) \bigg]\\
&\qquad\qquad\lesssim \max\{ 2^{j-|\ell|},  2^{2j-2|\ell|} (2^{-j+|\ell|}+2^{-|\ell|/10})     \} \times P_{j+\ell}\left(y,y_{d+1} \right).    \end{align*}
Thus
\begin{align*}
  \int |U_{j,\ell}(y,y_{d+1})|dy dy_{d+1}& \lesssim  2^{j-|\ell|}\int  |P_{j+\ell}\left(y,y_{d+1} \right)|dydy_{d+1} \int |K_{j,\ell}(z,z_{d+1})| dzdz_{d+1}\\
  &\lesssim  2^{j}2^j 2^{-|\ell|}=O(2^{-|\ell|/2}).
  \end{align*}
By this and (\ref{100302}) with $|\ell|>100d j$, we obtain (\ref{100301}).
\end{proof}

\end{document}